\documentclass[11pt,reqno]{amsart}

\usepackage{color}
\usepackage{amsmath}
\usepackage{amsfonts,amssymb}
\usepackage[normalem]{ulem}

\numberwithin{equation}{section}
\usepackage{mathrsfs}
\usepackage{amsthm}
\usepackage{t1enc , graphicx}
\newcommand{\R}{\mathbb{R}}
\newcommand{\C}{\mathbb{C}}
\newcommand{\T}{\mathbb{T}}

\newcommand{\Z}{\mathbb{Z}}
\newcommand{\N}{\mathbb{N}}

\newcommand{\ie}{scattered}
\newcommand{\Ie}{Scattered}

\newcommand{\Leb}{\mu}

\usepackage[pagebackref,colorlinks=true,urlcolor=blue,linkcolor=blue,citecolor=blue]{hyperref}
\usepackage[capitalize]{cleveref}
\newtheorem{theorem}{Theorem}[section]

\newtheorem{definition}[theorem]{Definition}
\newtheorem{example}[theorem]{Example}
\newtheorem{proposition}[theorem]{Proposition}
\newtheorem{question}[theorem]{Question}
\newtheorem{conjecture}[theorem]{Conjecture}

\newtheorem{remark}[theorem]{Remark}
\newtheorem{lemma}[theorem]{Lemma}
\newtheorem{corollary}[theorem]{Corollary}
\theoremstyle{plain}
\newtheorem*{namedthm}{\namedthmname}
\newcounter{namedthm}
\makeatletter
	\newenvironment{named}[2]
	{\def\namedthmname{#1}
	\refstepcounter{namedthm}
	\namedthm[#2]\def\@currentlabel{#1}}
	{\endnamedthm}
\makeatother
\newtheoremstyle{subproofstyle}{}{}{}{}{\bfseries}{:}{.5em}{}
\theoremstyle{subproofstyle}

\usepackage{etoolbox}
\AtEndEnvironment{subproof}{\hfill$\triangle$}

\definecolor{ggreen}{RGB}{0,200,0}
\definecolor{grey}{RGB}{200,150,150}
\usepackage[toc,page]{appendix}

\title{Metric uniform distribution on analytic curves
}

\author{Vitaly Bergelson and Joel Moreira
\\
With an appendix by Jim Wright}

\address[Vitaly Bergelson]{The Ohio State University, Columbus, OH, USA}		 \email{vitaly@math.ohio-state.edu}
\address[Joel Moreira]{University of Warwick, Coventry, UK}		 \email{joel.moreira@warwick.ac.uk}
\address[Jim Wright]{University of Edinburgh, Edinburgh, UK}		 \email{J.R.Wright@ed.ac.uk}

\renewcommand{\d}{~\mathrm{d}}
\begin{document}

\begin{abstract}
    We obtain multidimensional metric uniform distribution results involving sequences in $\R^k$ parametrized by analytic curves.
    Our theorems extend the classical theorems of Weyl and Koksma in a variety of ways.
    
    One of our main results implies that for any injective sequences $a_1,\dots,a_k:\N\to\Z$ the set 
    $$\Big\{(x_1,\dots,x_k)\in\R^k:\big(a_1(n)x_1,\dots,a_k(n)x_k\big)_{n\in\N}\text{ is uniformly distributed in }\T^k\Big\}$$
    has full Lebesgue measure inside any non-degenerate analytic curve $\gamma\subset\R^k$.
\end{abstract}

\maketitle

\tableofcontents

\section{Introduction}

\subsection{Background and description of the main results}

It was Borel who showed in \cite{Borel09} that for every natural number $b\geq2$, almost every real number $x$ is \emph{normal} in base $b$, meaning that each finite sequence of digits $w\in\{0,\dots,b-1\}^n$ appears in the $b$-adic expansion of $x$ with frequency $1/b^n$.
It was eventually revealed that Borel's theorem is equivalent to a special case (\cref{thm_Borelud} below) of a basic result in the theory of uniform distribution modulo one %\footnote{Indeed Borel theorem is equivalent, as was established in Wall's thesis, to the fact that the sequence $2^nx$ is uniformly distributed} 
established by Weyl in \cite{Weyl16}.
The goal of our paper is to provide an amplification of the aforementioned result of Weyl which deals with uniform distribution on analytic curves in higher dimensional tori.
Along the way we trace the historic development of the ideas triggered by Weyl's seminal paper.
We start with introducing the relevant definitions.
Throughout the paper we denote by $|X|$ the cardinality of the (finite) set $X$ and follow the convention that $\N=\{1,2,\dots\}$.

\begin{definition}[Uniform distribution]\label{def_uniformdistribution}
    Fix $k\in\N$.
    A sequence $(x_n)_{n\in\N}$ in $\R^k$ is \emph{uniformly distributed modulo 1} if for every intervals $I_1,\dots,I_k\subset[0,1]$,
    \begin{equation}\label{eq_def_uniformdistribution}
        \lim_{N\to\infty}\frac1N\Big|\big\{n\in\{1,\dots,N\}:\{x_n\}\in I_1\times\cdots\times I_k\big\}\Big|=\Leb (I_1\times \cdots\times I_k).
    \end{equation}
    where $\Leb$ denotes the $k$-dimensional Lebesgue measure and for $y=(y_1,\dots,y_k)\in\R^k$ we write $\{y\}$ for the vector of fractional parts $(\{y_1\},\dots,\{y_k\})$.
\end{definition}

Since the uniform distribution modulo 1 of a sequence $(x_n)_{n\in\N}$ in $\R^k$ only depends on the values of $\pi(x_n)$, where $\pi:\R^k\to\T^k=\R^k/\Z^k$ is the natural projection we will often say that $(x_n)_{n\in\N}$ is \emph{u.d. in $\T^k$} instead of \emph{uniformly distributed modulo 1}.

Along with the definition, in \cite{Weyl16} Weyl also introduced (and effectively used) the following criterion for uniform distribution, which still remains a fundamental tool for establishing a wide range of results in the theory of uniform distribution.
\begin{theorem}[Weyl's criterion]\label{thm_weylcriterion}
    Let $k\in\N$ and let $(x_n)_{n\in\N}$ be a sequence in $\R^k$.
    Then $(x_n)$ is uniformly distributed mod 1 if and only if for every non-zero $v\in\Z^k$,
    $$\lim_{N\to\infty}\frac1N\sum_{n=1}^Ne(v\cdot x_n)=0,$$
    where $v\cdot x$ denotes the usual scalar product in $\R^k$.
\end{theorem}

One can show that Borel's theorem is equivalent to the following statement pertaining to uniform distribution modulo one. 
\begin{theorem}\label{thm_Borelud}
    For each $b\geq2$, for almost every $x\in\R$ (with respect to Lebesgue measure) the sequence $(b^nx)_{n\in\N}$ is uniformly distributed modulo one.
\end{theorem}
The equivalence between Borel's theorem and \cref{thm_Borelud} was first observed by Wall in his thesis \cite[Theorem 1]{Wall50} (the proof can be found also in \cite[Chapter 1, Theorem 8.1]{Kuipers_Niederreiter74}).
\cref{thm_Borelud} is a rather special case of the following more general result obtained by Weyl in \cite{Weyl16}.

\begin{theorem}\label{thm_Weyl_intro}
    For any sequence $a:\N\to\R$ satisfying 
\begin{equation}\label{eq_thm_Weyl_intro}
    \inf_n\big(a(n+1)-a(n)\big)>0
\end{equation} 
the sequence $\big(a(n)x\big)_{n\in\N}$ is uniformly distributed modulo one for almost every $x\in\R$.
\end{theorem}

We stress that the special case in \cref{thm_Borelud} can also be derived from the pointwise ergodic theorem, but the full strength of \cref{thm_Weyl_intro} does not fit in an ergodic theoretic framework.
In fact, in \cite{Weyl16} Weyl proved a stronger variant of \cref{thm_Weyl_intro} (see \cref{thm_weylreals} below) that requires only a weaker form of \eqref{eq_thm_Weyl_intro}.
In this paper we consider a different weaker form of \eqref{eq_thm_Weyl_intro} and call a sequence satisfying it \emph{\ie}.
The precise definition of \ie{} sequence is postponed until \cref{def_ie}, but the reader may safely think of it as an averaged form of \eqref{eq_thm_Weyl_intro} that allows for a wider class of examples; a full discussion on \ie{} sequences, and their connection to both condition \eqref{eq_thm_Weyl_intro} and Weyl's original condition is presented in \cref{sec_iesequences}.

In this paper we are interested in ``non-trivial'' multidimensional extensions of Weyl's theorem.
We first note that it follows easily from \cref{thm_Weyl_intro} and Weyl's criterion that given a sequence $a:\N\to\R$ satisfying \eqref{eq_thm_Weyl_intro}, for almost every point $(x,y)\in\R^2$ the sequence $\big(a(n)x,a(n)y\big)_{n\in\N}$ is uniformly distributed in $\T^2$  (cf. the stronger \cref{prop_2dimKoksma} below).
This fact is related to a more general theorem obtained by Philipp in \cite{Philipp64}, stating that for any sequence $a(n)$ of $d\times d$ non-singular matrices with integers entries and $\det(a(n)-a(m))\neq0$ for all $n\neq m$, for almost every vector $x=(x_1,\dots,x_d)\in\T^d$ the sequence $(a(n)x)_{n\in\N}$ is uniformly distributed in $\T^d$.
%\jpm{Mention Philipp's result}

The goal of this paper is to address the more subtle question of whether given a (sufficiently regular) curve $\gamma$ in $\R^n$ one can have a Weyl-type theorem that applies to almost every point in $\gamma$.
For instance, if $a:\N\to\R$ satisfies \eqref{eq_thm_Weyl_intro}, is $\big(a(n)x,a(n)x^2\big)_{n\in\N}$ uniformly distributed in $\T^2$ for almost every $x\in\R$?
More generally, assume that $p_1,p_2\in\R[x]$ and $a,b:\N\to\R$; under which conditions (on $a,b,p_1,p_2$) is the sequence $\big(a(n)p_1(x),b(n)p_2(x)\big)_{n\in\N}$ uniformly distributed in $\T^2$ for almost every $x\in\R$?
Observe that if $p_1=p_2$ and $a=b$ then the sequence takes values in the diagonal $\{(x,x):x\in\T\}$ and hence cannot be uniformly distributed in $\T^2$.
The following result gives an answer to some of these questions. 

\begin{theorem}
\label{thm_muldidimpolWeyl1}
    Let $k\in\N$, let $p_1,\cdots, p_k\in\R[x]$. Then the following are equivalent:
    \begin{enumerate}
        \item The set $\{1,p_1,\dots,p_k\}$ is linearly independent over $\R$.
        \item For every \ie{} sequences $a_1,\dots,a_k:\N\to\R$, for Lebesgue-a.e.~$x\in\R$, the sequence $\big(a_1(n)p_1(x),\dots,a_k(n)p_k(x)\big)_{n\in\N}$ is uniformly distributed in $\T^k$.
    \end{enumerate}
\end{theorem}

\cref{thm_muldidimpolWeyl1} follows from the stronger \cref{thm_mainmultidimweyl1}. 
It implies, for example, that the sequence $\big(a(n)x,a(n)x^2,\dots,a(n)x^k\big)_{n\in\N}$ is u.d. mod 1 for almost every $x\in\R$ (whenever $a$ is \ie).
However, this result does not guarantee that the same conclusion holds for the sequence $\big(a(n)x,a(n)^2x, \dots,a(n)^kx\big)_{n\in\N}$.
The following theorem shows that, rather than assuming that the polynomials are linearly independent, one can instead impose an independence condition on the sequences $a_i(n)$.
Given sequences $a_1,\dots,a_k:\N\to\R$, we say that they are \emph{jointly \ie} if for every non-zero $v=(v_1,\dots,v_k)\in\R^k$ the linear combination $v_1a_1+\cdots+v_ka_k$ is \ie.

\begin{theorem}
\label{thm_muldidimpolWeyl2}
    Let $k\in\N$, let $p_1,\cdots, p_k\in\R[x]$ be non-constant polynomials and let $a_1,\dots,a_k:\N\to\R$ be jointly \ie{}.
    Then for Lebesgue-a.e.~$x\in\R$, the sequence $\big(a_1(n)p_1(x),\dots,a_k(n)p_k(x)\big)_{n\in\N}$ is uniformly distributed in $\T^k$.
\end{theorem}

\cref{thm_muldidimpolWeyl2} implies, for example, that for any \ie{} sequence $a:\N\to\R$, the sequence $\big(a(n)x,a(n)^2x, \dots,a(n)^kx\big)_{n\in\N}$ is u.d. mod 1 for almost every $x\in\R$.
Actually, if $a(n)$ is \ie, then $P\circ a$ is \ie{} for a large class of functions $P$ which includes all non-constant polynomials; see \cref{lemma_iecomposition}.

%It turns out that the two preceding theorems apply even if one replaces polynomials with analytic functions. 
The following theorem, which is proved in \cref{sec_proofofmainthm}, extends the scope of \cref{thm_muldidimpolWeyl1,thm_muldidimpolWeyl2} by replacing polynomials with general analytic functions.

\begin{theorem}\label{thm_mainmultidimweyl3}
        Let $k\in\N$, let $f_1,\dots,f_k:\R\to\R$ be non-constant analytic functions and let $a_1,\dots,a_k:\N\to\R$ be \ie{}.

        Assume that either
    \begin{enumerate}
        \item\label{part1_thm_muldidimanalWeyl} The sequences $a_1,\dots,a_k$ are jointly \ie{},
        
        \item[or] 
        \item\label{part2_thm_muldidimanalWeyl} The functions $\{1,f_1,\dots,f_k\}$ are linearly independent over $\R$.
    \end{enumerate}
    Then for a.e.~$x\in\R$, the sequence $\big(a_1(n)f_1(x),\dots,a_k(n)f_k(x)\big)_{n\in\N}$ is uniformly distributed in $\T^k$.
    \end{theorem}

%One can interpret \cref{thm_mainmultidimweyl3} in a more geometrical light.
\begin{remark}
    The condition that the $f_i$ are analytic in \cref{thm_mainmultidimweyl3} cannot be relaxed to $f_i\in C^\infty$. Indeed there exist non-constant $C^\infty$ functions such that $f(x)=2$ for a positive measure set of $x$.
    So, if $a(n)=n$ then $a(n)f(x)$ is an integer for every $n\in\N$ and every $x$ in a positive measure set, and hence cannot be u.d. mod 1 for almost every $x$.
    For more discussion on possible relaxations of the analytic condition see \cref{sec_weylfunctions} below. 
\end{remark}

Here is an equivalent formulation of \cref{thm_mainmultidimweyl3} in geometric terms.
An \emph{analytic curve} $\gamma\subset\R^k$ is the image of a real analytic function $f:\R\to\R^k$. It is \emph{non-degenerate} if it is not contained in a proper hyperplane. 
\begin{theorem}\label{thm_curvesformulation}
    Let $k\in\N$, $\gamma\subset\R^k$ be an analytic curve and let $a_1,\dots,a_k:\N\to\R$ be \ie.
    Assume that either the sequences $a_1,\dots,a_k$ are jointly \ie{} or that the curve is non-degenerate.
    Then for almost every $(x_1,\dots,x_k)\in\gamma$, the sequence $(a_1(n)x_1,\dots,a_k(n)x_k)$ is uniformly distributed in $\T^k$.
\end{theorem}

We remark that \cref{thm_curvesformulation} admits a ``multiparameter'' variant where the curve $\gamma$ is replaced by an analytic manifold in $\R^k$  of any dimension. 
This variant follows easily from \cref{thm_curvesformulation} via a Fubini-type argument.

In \cite{Koksma35} Koksma proved an extension of \cref{thm_Weyl_intro} that implies that the sequence $(x^{a(n)})_{n\in\N}$ is u.d. mod 1 for almost every $x>1$, provided that $a:\N\to\R$ satisfies \eqref{eq_thm_Weyl_intro}.
Our next result is a joint generalization of this fact and \cref{thm_muldidimpolWeyl1,thm_muldidimpolWeyl2}:

\begin{theorem}\label{thm_polKoksmaWeyl}

        Let $g:\R\to(1,\infty)$ be a nonconstant analytic function and let $b:\N\to\R$ tend to $\infty$ and be \ie{}.
        Let $k\in\N$, let $p_1,\cdots, p_k\in\R[x]$ be non-constant polynomials and let $a_1,\dots,a_k:\N\to\R$ be \ie{}.
        Assume that either
    \begin{enumerate}
        \item\label{part1_thm_muldidimanalWeyl} The sequences $a_1,\dots,a_k$ are jointly \ie{}, 
        
        \item[or] 
        \item\label{part2_thm_muldidimanalWeyl} The polynomials $\{1,p_1,\dots,p_k\}$ are linearly independent.
    \end{enumerate}
        Then, for almost every $x\in\R$, the sequence $\big(g(x)^{b(n)},a_1(n)p_1(x),\dots,a_k(n)p_k(x)\big)$ is u.d. in $\T^{k+1}$.
\end{theorem}

It seems reasonable to expect that an extension of \cref{thm_polKoksmaWeyl}, in a way similar to how \cref{thm_mainmultidimweyl3} extends \cref{thm_muldidimpolWeyl1,thm_muldidimpolWeyl2}, should hold, where polynomials are replaced by general analytic functions. 
Unfortunately, our techniques do not seem sufficient to obtain such a result.
In this direction, we make the following conjecture (see also \cref{sec_multikoksma} for other potential extensions of \cref{thm_polKoksmaWeyl}).

\begin{conjecture}\label{conj_koksma}
\cref{thm_polKoksmaWeyl} holds when $p_1,\dots,p_k$ are general nonconstant analytic functions.
\end{conjecture}

In support of \cref{conj_koksma} we have the following result, corresponding to the case $k=1$, which can be obtained as a corollary of \cref{thm_polKoksmaWeyl} via a ``change of variable'' maneuver.
\begin{corollary}\label{thm_weyl_koksmafusion}
        Let $f,g:\R\to\R$ be non-constant analytic functions such that $g(\R)\subset(1,\infty)$ and let $a,b:\N\to\R$ be \ie{} sequences. 
        Then for almost every $x\in\R$ the sequence $\big(g(x)^{b(n)},a(n)f(x)\big)$ is u.d. on $\T^2$.
\end{corollary}

We conclude this section with the formulation of a general question which encompasses all the results formulated above.

\begin{question}\label{question_KoksmaWeyl2k}
    Let $k\in\N$, $f_1,\dots,f_k:\R\to\R$ and $g_1,\dots,g_k:\R\to(1,\infty)$ be analytic functions and let $a_1,\dots,a_k,b_1,\dots,b_k:\N\to\R$ be \ie.
    Under which conditions on $f_i,g_i,a_i,b_i$ is it true that for almost every $x\in\R$ the sequence
    $$\Big(g_1(x)^{b_1(n)}\ ,\ \dots\ ,\ g_k(x)^{b_k(n)}\ ,\ a_1(n)f_1(x)\ ,\ \dots\ ,\ a_k(n)f_k(x)\Big)_{n=1}^\infty$$
    is uniformly distributed on $\T^{2k}$?
\end{question}

\subsection{Structure of the paper}
The paper is organized as follows: \cref{sec_expository} is a partly expository section dedicated to an introduction to metric equidistribution theorems following the footsteps of Weyl and Koksma. In particular, we provide modern proofs of Weyl's theorem and Koksma's theorem; showcase the role of asymptotic orthogonality and demonstrate how it leads to amplifications of a classical theorem of Raikov; introduce and explore the notion of \ie{} sequences, including comparison with Weyl's growth condition and finally prove \cref{thm_muldidimpolWeyl2}.

In \cref{sec_proofofmainthm} we first prove \cref{thm_mainmultidimweyl1}, which extends the scope of \cref{thm_muldidimpolWeyl1}. 
Its proof combines the arguments used to derive \cref{thm_muldidimpolWeyl2} with the following oscillatory integral estimate, due to J. Wright, the proof of which is given in the appendix.
\begin{theorem}\label{thm_wright}
            Let $k\in\N$, let $f_1,\dots,f_k:\R\to\R$ be analytic functions such that $\{1,f_1,\dots,f_k\}$ is linearly independent over $\R$, and let $I\subset\R$ be a compact interval.
            
            Then there exist $C,\delta>0$ such that for every non-zero $\lambda=(\lambda_1,\dots,\lambda_k)\in\R^k$,
            \begin{equation} \label{eq_wrightestimate}
            \left|\int_Ie\big(\lambda_1f_1(x)+\cdots+\lambda_kf_k(x)\big)\d x\right|\leq C\|\lambda\|^{-\delta}.
            \end{equation}
        \end{theorem}

The second half of \cref{sec_proofofmainthm} contains a proof of \cref{thm_mainmultidimweyl3}, which builds upon \cref{thm_mainmultidimweyl1} but requires a more delicate analysis.
Lastly, in \cref{sec_koksmaweyl} we prove \cref{thm_polKoksmaWeyl} and discuss further potential extensions of Koksma's theorem to the multidimensional setting.
        
\subsection{Acknowledgments}

Part of the research contained in this paper was conducted while the authors were visiting the Institute for Advanced Studies at Princeton during parts of the 2022/23 academic year.
We are grateful to the Institute for its hospitality and support which was partly funded by the NSF grant DMS-1926686. 
The second author is supported by the EPSRC Frontier Research Guarantee grant
EP/Y014030/1. 
For the purpose of open access, the authors have applied a Creative Commons Attribution (CC-BY) license to any Author Accepted Manuscript version arising from this submission.

\section{A survey of metric uniform distribution}
\label{sec_expository}

Given a sequence $(u_n(x))_{n=1}^\infty$ depending on a parameter $x\in\R$, we are interested in conditions which guarantee that for almost every value of $x$ the sequence is uniformly distributed.
The first result in this direction, \cref{thm_Weyl_intro}, was obtained by Weyl in the pioneering work \cite{Weyl16}, thereby starting the field of metric uniform distribution.

In this section we present two fundamental theorems in this area, due to Weyl and Koksma, and present modern proofs. 
%We explore how metric uniform distribution connects to some 
We pay particular attention to the quantitative aspects of the assumption in Weyl's theorem, and define the notion of \ie{} sequences, which slightly extends its scope.
Throughout the section we adapt the relatively simple proof of \cref{thm_Weyl_intro} (presented in subsection \ref{sec_survey1}) to handle progressively more general results, culminating in subsection \ref{sec_surveylast} with a proof of \cref{thm_muldidimpolWeyl2}, a new multidimensional result that involves polynomial curves.

\subsection{Metric uniform distribution -- first results}
\label{sec_survey1}

We start this section with the following \cref{thm_weylintegers} which is a special case of \cref{thm_Weyl_intro} but already implies Borel's theorem (\cref{thm_Borelud}).
The proof of \cref{thm_weylintegers}, albeit rather simple, provides a useful prototype for the proofs of stronger results obtained later in this paper.

\begin{theorem}\label{thm_weylintegers}
   Let $(a_n)$ be an injective sequence of integers. Then for (Lebesgue) almost every $x\in\R$ the sequence $(a_nx)$ is uniformly distributed modulo 1.
\end{theorem}
\begin{proof}
    
In view of Weyl's criterion (\cref{thm_weylcriterion}), which in the one dimensional case states that $(x_n)_{n\in\N}$ is uniformly distributed mod 1 if and only if for every $k\in\N$ the sequence $\frac1N\sum_{n=1}^Ne(kx_n)$ tends to $0$ as $N\to\infty$, it suffices to show that for every $k\in\N$ the sequence of functions $F_N(x):=\frac1N\sum_{n=1}^Ne(ka_nx)$ converges to $0$ almost everywhere.
Since each function $F_N$ is periodic with period one, we may restrict its domain to $[0,1)$.
Using the general fact that $\left|\frac1N\sum_{n=1}^N u_n -\frac1M\sum_{n=1}^M u_n \right|\leq2(1-N/M)$ whenever $N<M$ and $(u_n)$ is a sequence of complex numbers with $|u_n|=1$, it is enough to show that $F_{N^2}(x)\to0$ as $N\to\infty$ for almost every $x\in[0,1)$.
This is obviously equivalent to the statement that $\big|F_{N^2}(x)\big|^2\to0$ as $N\to\infty$ for almost every $x\in\R$.
As any summable sequence must tend to zero, we need only to show that $F(x):=\sum_N|F_{N^2}(x)|^2<\infty$ for almost every $x\in[0,1)$.
To establish this fact, it suffices to show that the integral $\int_0^1F(x)\d x$ is finite. 
Using the monotone convergence theorem we get
\begin{equation}\label{eq_before4}
    \int_0^1F(x)\d x
    =
    \sum_{N=1}^\infty\int_0^1\big|F_{N^2}(x)\big|^2\d x
    =
    \sum_{N=1}^\infty\big\|F_{N^2}(x)\big\|^2_{L^2[0,1]}
\end{equation}
    
and, using orthogonality of characters (together with the fact that $a_n\neq a_m$ whenever $n\neq m$),
\begin{equation}\label{eq_orthogonalityofcharacters}
    \big\|F_{N}(x)\big\|^2_{L^2[0,1]}
=
\frac1{N^2}\sum_{n,m=1}^N\big\langle e(k a_n x),e(k a_mx)\big\rangle_{L^2[0,1]}
=
\frac1N.
\end{equation}
Equations \eqref{eq_before4} and \eqref{eq_orthogonalityofcharacters} together show that $\int_0^1F(x)\d x=
    \sum_{N=1}^\infty\big\|F_{N^2}(x)\big\|^2_{L^2[0,1]}=
    \sum_{N=1}^\infty1/N^2<\infty$, finishing the proof.
\end{proof}

\cref{thm_weylintegers} also holds for sufficiently spread-out sequences of real numbers. 
Before addressing the full version of this extension in \cref{thm_weylreals} below, we discuss first a special case that suffices to recover \cref{thm_Weyl_intro} stated in the introduction.
%Note that \cref{thm_weylrealsbaby} has no requirement on the monotonicity of the sequences involved.

\begin{theorem}[Weyl, 1916]\label{thm_weylrealsbaby}
    Let $(a_n)$ be a sequence of real numbers and suppose that there exists a gap $g>0$ satisfying $|a_n-a_m|\geq g$ for all $n\neq m$. 
    Then for almost every $x\in\R$ the sequence $(a_nx)$ is uniformly distributed modulo 1.
\end{theorem}

%One can prove \cref{thm_weylrealsbaby} by modifying the proof of \cref{thm_weylintegers} presented above with only minor modifications.
The proof of \cref{thm_weylrealsbaby} requires only minor modifications to the proof of \cref{thm_weylintegers} presented above.
%The main difference when the $a_n$ are not integers is that the functions $x\mapsto e(k a_n x)$ are not characters on $\R/\Z$ and hence are not pairwise orthogonal, invalidating the computation in \eqref{eq_orthogonalityofcharacters}.
When the $a_n$ are not integers, the functions $x\mapsto e(k a_n x)$ are not characters on $\R/\Z$ and hence are not pairwise orthogonal, invalidating the computation in \eqref{eq_orthogonalityofcharacters}.
Nevertheless, it turns out that when $a-b$ is sufficiently large, the functions $x\mapsto e(k a x)$ and $x\mapsto e(k b x)$ are ``almost orthogonal''.
This \emph{asymptotic orthogonality} allows us to estimate the $L^2$ norm of $F_N(x)=\frac1N\sum_{n=1}^Ne(k a_n x)$ and is formalized using the following inequality\footnote{ \cref{lemma_vdCestimate}, which was proved by van der Corput in 1921, was not available to Weyl in 1916. 
As we will see, \cref{lemma_vdCestimate} streamlines the proofs and plays an important role in the rest of our paper.
Weyl's approach to proving \cref{thm_weylrealsbaby}, and the stronger \cref{thm_weylreals} formulated in \cref{sec_weylgrowth} below, uses instead a clever approximation technique that does not require \cref{lemma_vdCestimate} but does not seem to work in more general situations. 
%Nevertheless, 
} due to van der Corput \cite{vdCorput21} (we use here the version from {\cite[Lemma 2.1, page 15]{Kuipers_Niederreiter74}}).
\begin{lemma}
\label{lemma_vdCestimate}
%\jpm{find original reference?}

    Let $f:[a,b]\to\R$ be differentiable and such that $f'$ is continuous and monotone. 
    Then
    $$\left|\int_a^be\big(f(x)\big)\d x\right|%<\frac{1}{\inf\big\{|f'(x)|:x\in I\big\}}=
    \leq\max\left(\frac1{|f'(a)|},\frac1{|f'(b)|}\right).$$
\end{lemma}

Using \cref{lemma_vdCestimate} it follows that\footnote{Here and elsewhere in the paper we use the notation $A\ll B$ to denote $A\leq CB$ for some constant $C>0$.} 
\begin{equation}\label{eq_firstafter2.2}
    \big|\big\langle e(k a_n \cdot),e(k a_m\cdot)\big\rangle_{L^2}\big|=\left|\int e( k (a_n-a_m)x)\d x\right|\ll\frac1{|a_n-a_m|}.
\end{equation}
%\jpm{The notation $\ll$ needs to be introduced.}
Since the set of real numbers $A:=\{a_n:n\in\N\}$ has any pair of points at distance at least $g>0$, it follows that
$\sum_{n=1}^{m-1}\frac1{|a_n-a_m|}\leq\sum_{n=1}^{m-1}\frac2{g n}\ll\log m$ and hence
\begin{equation}\label{eq_secondafter2.2}
    \sum_{1\leq n<m\leq N}\frac1{|a_n-a_m|}\ll N\log N.
\end{equation}
Combining \eqref{eq_firstafter2.2} and \eqref{eq_secondafter2.2}, it follows that the function $F(x):=\sum_N|F_{N^2}(x)|^2$ satisfies 
$\int_a^bF(x)\d x=
    \sum_{N=1}^\infty\big\|F_{N^2}(x)\big\|^2_{L^2[a,b]}\ll
    \sum_{N=1}^\infty(\log N/N)^2<\infty$, finishing the proof of \cref{thm_weylrealsbaby}.

A far reaching extension of \cref{thm_weylrealsbaby} was obtained by Koksma in \cite[Satz 3]{Koksma35} (see \cite[Theorem 4.3, page 34]{Kuipers_Niederreiter74}).

\begin{theorem}[Koksma]\label{thm_koksmafull}
    Let $a<b$ and, for each $n\in\N$, let $u_n\in C^1([a,b],\R)$.
    Assume that there exists $g>0$ such that whenever $n\neq m$ the function $x\mapsto\big|u'_n(x)-u_m'(x)\big|$ is monotone and bounded from below by $g$.

    Then for almost every $x\in[a,b]$ the sequence $\big(u_n(x)\big)_{n\in\N}$ is uniformly distributed mod 1.
\end{theorem}

Note that \cref{thm_koksmafull} implies \cref{thm_weylrealsbaby} when applied to linear functions $u_n(x)=a_nx$.
Another famous corollary of \cref{thm_koksmafull} is that the sequence $(x^n)_{n\in\N}$ is uniformly distributed for almost every $x>1$, which can be derived from \cref{thm_koksmafull} by considering $u_n(x)=x^n$.

One can prove \cref{thm_koksmafull} by following the same basic steps as the proof of \cref{thm_weylrealsbaby} outlined above.
In order to obtain an appropriate analogue of \eqref{eq_orthogonalityofcharacters} one utilizes \cref{lemma_vdCestimate}. 
More precisely, denoting by $F_N:[a,b]\to\R$ the function $F_N(x)=\frac1N\sum_{n=1}^Ne(ku_n(x))$, it follows from \cref{lemma_vdCestimate} that 
\begin{eqnarray*}
    \|F_N\|^2_{L^2[a,b]}
&\leq&
\frac1{N^2}\sum_{n,m=1}^N\max\left(\frac1{|u_n'(a)-u_m'(a)|},\frac1{|u_n'(b)-u_m'(b)|}\right)
\\&\leq&
\frac1{N^2}\sum_{n,m=1}^N\frac1{|u_n'(a)-u_m'(a)|}+\frac1{N^2}\sum_{n,m=1}^N\frac1{|u_n'(b)-u_m'(b)|}.
\end{eqnarray*}
Since both sets $\{u_n'(a):n\in\N\}$ and $\{u_n'(b):n\in\N\}$ are $g$-separated sets of real numbers, it follows (using the same reasoning as in obtaining \eqref{eq_secondafter2.2}) that $\|F_N\|^2_{L^2(I)}\ll \log{N}/N$.
The conclusion of the proof now is the same as for \cref{thm_weylrealsbaby}.
By sharpening various elements of the above proof one can obtain the more general %\cref{thm_koksmafull} is a special case of more general 
\cref{thm_betterkoksma} which is formulated and proved in \cref{sec_koksmarevisited}.
%\cref{lemma_vdCestimate} is a special case of \cref{thm_betterkoksma} whose full proof is presented below.

\subsection{A more general view of \cref{thm_weylintegers}}

An equivalent form of the definition of uniform distribution for a sequence $(x_n)_{n\in\N}$ of real numbers is that for every continuous function $f\in C(\T)$,
$$\lim_{N\to\infty}\frac1N\sum_{n=1}^Nf(\{x_n\})=\int_0^1 f(x)\d x.$$
A similar characterization of uniform distribution for sequences in $\R^d$ can be formulated in the obvious way (these equivalent forms were already present in \cite{Weyl16}).
In light of this fact, we can rephrase \cref{thm_weylintegers} as follows: 
\begin{equation}\label{eq_weylfunctionalversion}
    \forall f\in C(\T)\qquad\frac1N\sum_{n=1}^Nf(\{a_nx\})\to\int f(x)\d x \quad a.s.
\end{equation}
While \eqref{eq_weylfunctionalversion} resembles an ergodic theorem, a closer look at the proof of \cref{thm_weylintegers} given above reveals that the same statement holds when the averaging sets $\{1,\dots,N\}$ get replaced by \emph{any} (not necessarily F\o lner) sequence of strictly nested finite sets in $\N$, leading to the following ostensibly stronger result.
\begin{theorem}\label{theorem_weylanyaverages}
    Let $(a_n)_{n\in\N}$ be an injective sequence of integers and, for each $N\in\N$, let $S_N\subset\N$ be a finite set such that $S_N\subsetneq S_{N+1}$.
    Then
    \begin{equation*}%\label{eq_weylfunctionalversion}
    \forall f\in C(\T)\qquad\frac1{|S_N|}\sum_{n\in S_N}f(\{a_nx\})\to\int f(x)\d x \quad a.s.
\end{equation*}
\end{theorem}

A similar theorem with the same conclusion but different assumption on the $(S_N)$ can also be easily derived:
%In fact, one can replace the nested condition on the sets $(S_N)_{N\in\N}$, with a mild growth condition as follows. 
\begin{theorem}\label{thm_hoeffding}
    Let $(a_n)_{n\in\N}$ be an injective sequence of integers and, for each $N\in\N$, let $S_N\subset\N$ be a finite set such that $\sum|S_N|^{-1}<\infty$.
    Then
    \begin{equation*}\label{eq_weylfunctionalversion3}
    \forall f\in C(\T)\qquad\frac1{|S_N|}\sum_{n\in S_N}f(\{a_nx\})\to\int f(x)\d x \quad a.s.
\end{equation*}
\end{theorem}

\begin{proof}
    Following the proof of \cref{thm_weylintegers}, for a fixed $k\in\N$ define the function $F_N(x):=|S_N|^{-1}\sum_{n\in S_N}e(ka_nx)$.
    As in \eqref{eq_orthogonalityofcharacters}, orthogonality of characters yields $\|F_N\|_{L^2[0,1]}^2=|S_N|^{-1}$ and hence the growth assumption implies that $\sum\|F_N\|_{L^2([0,1])}^2<\infty$.
    
    We deduce that $|F_N(x)|\to0$ almost everywhere as $N\to\infty$, and the desired conclusion now follows from Weyl's criterion.
\end{proof}
We remark in passing that each of \cref{theorem_weylanyaverages,thm_hoeffding} implies a strengthening of a result of Raikov \cite{Raikov36} (who dealt with the case $S_N=\{1,\dots,N\}$), replacing continuous functions with more general $L^1$ functions, and convergence in norm with almost sure convergence:
% \begin{corollary}\label{cor_raikov}
% Let $(a_n)_{n\in\N}$ be an injective sequence of integers and, for each $N\in\N$, let $S_N\subset\N$ be a finite set such that either $S_N\subsetneq S_{N+1}$ or $\sum|S_N|^{-1}<\infty$.
%     Then
%     \begin{equation}\label{eq_raikov}
%     \forall f\in L^1(\T)\qquad\frac1{|S_N|}\sum_{n\in S_N}f(\{a_nx\})\to\int f(x)\d x \text{ in $L^1$-norm}. 
% \end{equation}
% \end{corollary}
% \begin{proof}
%     \cref{theorem_weylanyaverages,thm_hoeffding} imply that \eqref{eq_raikov} holds for continuous $f$.
%     Since $C(\T)$ is dense in $L^1(\T)$, the conclusion follows.
% \end{proof}

\begin{theorem}\label{thm_BBR}
Let $(a_n)_{n\in\N}$ be an injective sequence of integers and, for each $N\in\N$, let $S_N\subset\N$ be a finite set such that $|S_N|\to\infty$.
    Then
    \begin{equation*}%\label{eq_raikov}
    \forall f\in L^1(\T)\qquad\frac1{|S_N|}\sum_{n\in S_N}f(\{a_nx\})\to\int f(x)\d x \text{ in $L^1$-norm}. 
\end{equation*}
\end{theorem}

Another modern way to derive Raikov's theorem was presented in \cite[Section 2 (ii)]{Berend_Bergelson86}.

Note that in \cref{thm_BBR} the only assumption on $(S_N)$ is that the size goes to infinity, while for almost sure convergence for continuous functions we required some additional (but mild) condition on the sequence $(S_N)$.
%Interestingly, for Borel's theorem about normality (which is a special case of Weyl's theorem) it is again possible to drop any assumptions on $(S_N)$ other than the size goes to infinity.
%\jpm{Remark that some condition on the sets is necessary here, but not for Raikov below.}

\cref{thm_BBR} admits a multidimensional extension, which also strengthens the theorem of Philipp \cite[Theorem 2]{Philipp64}, mentioned in the introduction:%, which extends \cref{thm_weylintegers} to the setting of $d\times d$ matrices. 
%An even stronger result, which has \cref{thm_BBR} as a special case, can be derived.
\begin{theorem}
    Let $d\in\N$ and let $(A_n)_{n\in\N}$ be a sequence of non-singular $d\times d$ matrices with integer coefficients and such that $\det(A_n-A_m)\neq0$ whenever $n\neq m$.
    For each $N\in\N$, let $S_N\subset\N$ be a finite set such that $|S_N|\to\infty$ as $N\to\infty$.
    Then 
    \begin{equation*}%\label{eq_raikov}
    \forall f\in L^1(\T^d)\qquad\frac1{|S_N|}\sum_{n\in S_N}f\big(\{A_nx\}\big)\to\int_{\T^d} f(x)\d x \text{ in $L^1$-norm}. 
\end{equation*}
\end{theorem}
\begin{proof}
As finite linear combinations of characters are dense in $L^1(\T^d)$, we can assume that $f(x)=e(v\cdot x)$ is the character associated with some $v\in\Z^d\setminus\{0\}$.
    %For each $n\in\N$, let $f_n:x\mapsto f(A_nx)$. 
    For $m\neq n$ we have 
    $$\int_{\T^d}f(A_nx)\cdot\overline{f(A_mx)}\d x
    =
    \int_{\T^d}e\Big(v\cdot (A_n-A_m)x\Big)\d x.
    $$
Since $v\neq0$ and $\det(A_n-A_m)\neq0$, it follows that the integral vanishes and hence the functions $x\mapsto f(A_nx)$ with $n\in\N$ are pairwise orthogonal (and each has norm $1$).
This implies that 
$$\left\|\frac1{|S_N|}\sum_{n\in S_N}f(\{A_n x\})\right\|_{L^2}^2=\frac1{|S_N|}\to0\text{ as }N\to\infty,$$
which by Cauchy-Schwarz's inequality implies that the same occurs in $L^1$ norm and completes the proof.
\end{proof}
We conclude this subsection with a caveat. 
Weyl's theorem, as captured by \eqref{eq_weylfunctionalversion}, deals with almost sure convergence of continuous functions and it implies Raikov's theorem  which deals with norm convergence of $L^1$ functions.
This leads to a natural question: does the convergence in \eqref{eq_weylfunctionalversion} still hold (almost surely) for $f\in L^1(\T)$? 
A less ambitious question, asked by Khintchine in \cite{Khintchine23}, is whether \eqref{eq_weylfunctionalversion} holds when $f$ is the indicator function of a Borel set.
This was answered in the negative by Marstrand\footnote{In \cite{Khintchine23} Khintchine only asks a question, but judging by the title of \cite{Marstrand70}, Marstrand interpreted it as a conjecture.} in \cite{Marstrand70}, who produced a Borel set $A\subset\T$ with arbitrarily small measure such that 
$$\limsup_{N\to\infty}\frac1N\sum_{n=1}^N1_A(nx)=1$$
for all $x$, showing that an analogue of \eqref{eq_weylfunctionalversion} where $f\in L^\infty(\T)$ (let alone $f\in L^1(\T)$), does not hold.

\subsection{Weyl's growth condition}
\label{sec_weylgrowth}
In \cite{Weyl16} Weyl proved a result that applies to a more general class of sequences than that covered by \cref{thm_weylrealsbaby}.
\begin{definition}[Weyl's growth condition]\label{def_weylgrowth}
     A (not necessarily monotone) sequence $(a_n)$ of real numbers satisfies \emph{Weyl's growth condition} if there exist $\epsilon,g>0$ such that $|a_m-a_n|>g$ whenever $m>n+n/(\log n)^{1+\epsilon}$.
\end{definition}
\begin{remark}
    Note that when $(a_n)$ is a (not strictly) increasing sequence of integers, the growth condition becomes $a_n\neq a_m$ for $m-n>n/(\log n)^{1+\epsilon}$. 
    It was in this form that Weyl first formulated it in his paper.
\end{remark}

Any sequence satisfying the assumptions of \cref{thm_weylrealsbaby} trivially satisfies Weyl's growth condition. 
Here is, finally, Weyl's theorem in full generality.

\begin{theorem}[Weyl, {\cite{Weyl16}}]\label{thm_weylreals}
    If the sequence $(a_n)$ satisfies Weyl's growth condition, then for almost every $x\in\R$ the sequence $(a_n x)_{n\in\N}$ is uniformly distributed mod 1.
\end{theorem}
Below we prove \cref{thm_betterkoksma}, which contains \cref{thm_weylreals} as a special case. 
As will be seen, the proof of \cref{thm_betterkoksma} follows the same pattern as the proof of \cref{thm_koksmafull} given above, but with a somewhat more involved quantitative analysis.

\cref{thm_weylreals} applies to many natural sequences that do not satisfy the condition in \cref{thm_weylrealsbaby} above. Here is a natural family of examples:
\begin{example}\label{example_logcube}
    For every $\epsilon>0$, the sequence $a_n=(\log n)^{2+\epsilon}$ satisfies Weyl's growth condition.

    Indeed, if $m\geq2n$ then $|a_m-a_n|\geq\log 2$.
    For $m<2n$, using the mean value theorem we have $|a_m-a_n|=(m-n)(2+\epsilon)(\log x)^{1+\epsilon}/x$ for some $n\leq x\leq m$.
    If $m-n>n/(\log n)^{1+\epsilon'}$ for some $\epsilon'<\epsilon$, since $x\leq m< 2n$, it follows that $|a_m-a_n|\geq(\log n)^{\epsilon-\epsilon'}\gg1$. Hence $(a_n)_{n=1}^\infty$ satisfies Weyl's growth condition.
\end{example}
It is worth mentioning that \cref{example_logcube} is sharp in the sense that the sequence $a_n=(\log n)^2$ does \emph{not} satisfy Weyl's growth condition, see \cref{example_logsquare} below.
On the other hand, a theorem of Fej\'er (see, e.g., \cite[Theorem 2.5]{Kuipers_Niederreiter74}) states that if $(a_n)$ is a sequence whose discrete derivative $\Delta a_n:=a_{n+1}-a_n$ is decreasing and tends to zero and $n\cdot\Delta a_n$ tends to infinity, then $a_n$ is uniformly distributed mod 1.
It follows that the sequence $a_n=(\log n)^{1+\epsilon}$ satisfies the conclusion of Weyl's theorem for every $\epsilon>0$ (actually, in this case $(a_nx)_{n\in\N}$ is u.d. for \emph{every} nonzero $x\in\R$). 
Fej\'er's theorem is complemented by the fact, proved by Dress \cite{Dress68}, that if $(a_n)$ is an increasing sequence satisfying $a_n=o(\log n)$, then for every real number $x$ the sequence $(a_n x)$ is \emph{not} uniformly distributed mod 1.
There is a curious discrepancy between the $(\log n)^{1+\epsilon}$ that is covered by F\'ejer's theorem and the $(\log n)^{2+\epsilon}$ that is guaranteed by Weyl's theorem; it would be desirable to better understand if it is possible to formulate a theorem in the spirit of Weyl's which applies to the sequence $a_n=(\log n)^{1+\epsilon}$ for $\epsilon\in(1,2]$.

Notice that Weyl's growth condition isn't exactly about the growth of the sequence but about how nearby terms clump together. 
For instance the sequence $a_n=\exp(\exp(\lfloor \log n\rfloor))$ satisfies $a_n\geq e^{n-1}$ but it does not satisfy Weyl's growth condition (nor the conclusion of \cref{thm_weylreals}).
As pointed out by Ruzsa in \cite{Ruzsa83}, the property that allows the proof of Weyl's theorem to run is that there are not too many pairs $(n,m)$ with $|a_n-a_m|<1$; Ruzsa then shows in \cite[Theorem 3]{Ruzsa83} that in that sense Weyl's condition is optimal.

\subsection{\Ie{} sequences}\label{sec_iesequences}
In this section we give the precise definition of \ie{} sequences, which is an averaged variant of Weyl's growth condition.
This notion was used in the formulations of our main theorems \cref{thm_mainmultidimweyl3} and \cref{thm_polKoksmaWeyl} in the introduction and will be used in \cref{sec_koksmarevisited} to give a streamlined proof of a theorem of Koksma.
%We start this subsection with the precise definition of the averaged variant of Weyl's growth condition that we use in this paper.
\begin{definition}
\label{def_ie}
   Fix $\delta\in(0,1]$.
A sequence $a:\N\to\R$ is called \emph{$\delta$-\ie{}} if there exists $\epsilon>0$ such that for every sufficiently large $N$,
    \begin{equation}\label{eq_iedef}
\frac1{N^2}\sum_{1\leq m<n\leq N}\min\Big(\big|a(n)-a(m)\big|^{-\delta},1\Big)\leq \frac1{(\log N)^{1+\epsilon}}.
    \end{equation}
We say that the sequence $a$ is \emph{\ie{}} if it is $\delta$-\ie{} for \emph{every} $\delta\in(0,1]$. 
\end{definition}

Observe that the property of being $\delta$-\ie{} gets stronger as $\delta$ decreases and so requiring a sequence to be $1$-\ie{} places the weakest restriction on it.
Below we will see that if a sequence satisfies Weyl's growth condition (see \cref{def_weylgrowth}), then it is $1$-\ie{}, but not the other way around.
In particular, \cref{thm_betterkoksma} below, which applies to 1-\ie{} sequences, contains \cref{thm_weylreals} as a special case. 
On the other hand, for any $\delta<1$ there are sequences that satisfy Weyl's growth condition but are not $\delta$-\ie.
Therefore the multidimensional theorems stated in the introduction, while being more subtle in many respects, require the stronger assumption that the sequences involved are $\delta$-\ie{} for every $\delta$, and hence do not reduce to \cref{thm_weylreals} in full generality when the dimension is 1.

The following lemma shows that composition with functions that grow linearly or faster preserves the property of being \ie{}.
\begin{lemma}\label{lemma_iecomposition}
    Suppose $a:\N\to\R$ is $\delta$-\ie{}, $P:\R\to\R$ is differentiable and there exists $c>0$ such that $|P'(x)|>c$ for every $|x|>1/c$. Then $P\circ a$ is $\delta$-\ie.
\end{lemma}

\begin{proof}
    Since $a$ is \ie, the set $Z:=\{n:|a(n)|\leq1/c\}$ has zero density. 
    The mean value theorem implies that $|P(a(n))-P(a(m))|\geq c|a(n)-a(m)|$ whenever $n,m$ are outside $Z$.
    It now follows from \eqref{eq_iedef} that
    $$\frac1{N^2}\sum_{1\leq m<n\leq N}\min\Big(\big|P(a(n))-P(a(m))\big|^{-\delta},1\Big)\leq\frac {c^{-\delta}}{(\log N)^{1+\epsilon}}+o_{N\to\infty}(1).$$
    For any $\epsilon'<\epsilon$, if $N$ is sufficiently large we have $\frac {c^{-\delta}}{(\log N)^{1+\epsilon}}\leq \frac1{2(\log N)^{1+\epsilon'}}$, finishing the proof.
\end{proof}

Our next lemma shows that if a sequence satisfies Weyl's growth condition then it is $1$-\ie.

\begin{lemma}\label{lemma_weylsgrowthsuffice}
    Let $a:\N\to\R$ and suppose that there exist $\epsilon,g>0$ such that $|a(n)-a(m)|>g$ whenever $m>n+n/(\log n)^{1+\epsilon}$.
    Then for any $\epsilon'<\epsilon$ and every sufficiently large $N$,
    \begin{equation}\label{eq_iedef6}
\frac1{N^2}\sum_{1\leq m<n\leq N}\min\left(|a(n)-a(m)|^{-1},1\right)\leq  \frac1{(\log N)^{1+\epsilon'}}.
    \end{equation}
    and in particular $a$ is $1$-\ie{}.
\end{lemma}

\begin{proof}

Given $N\in\N$ and an interval $I\subset\R$ of length $g$, we have that the set $\{n\leq N:a(n)\in I\}$ has at most $N/(\log N)^{1+\epsilon}$ elements; indeed, if $n_0$ is the smallest element of that set, the largest element is at most $n_0+n_0/\log(n_0)^{1+\epsilon}$.
It follows that for an arbitrary bounded interval $I\subset\R$, breaking it into sub-intervals of length at most $g$, we have
$$\Big|\big\{n\leq N:a(n)\in I\big\}\Big|\leq\left(\frac{|I|}g+1\right)\frac{N}{(\log N)^{1+\epsilon}}.$$

As a consequence, for every length $\ell>0$, the set
$$A_\ell=A_\ell(N):=\Big\{(n,m):1\leq n<m\leq N,\ \big|a(n)-a(m)\big|\leq\ell\Big\}$$
has cardinality
\begin{equation}\label{eq_proof_weyl_growth}
    |A_\ell|
\leq
\sum_{n=1}^N\Big|\big\{m\leq N:\big|a(m)-a(n)\big|<\ell\big\}\Big|
\leq 
N\cdot\frac{2\ell N}{g(\log N)^{1+\epsilon}}
=
\frac{2\ell N^2}{g(\log N)^{1+\epsilon}}.
\end{equation}
We can now write the left hand side of \eqref{eq_iedef6} as
\begin{eqnarray*}
    \frac1{N^2}\sum_{1\leq m<n\leq N}\min\left(|a(n)-a(m)|^{-1},1\right)
    &\leq&
    \frac1{N^2}\left(|A_1|+\sum_{\ell=1}^\infty\ell^{-1}\big(|A_{\ell+1}|-|A_{\ell}|\big)\right)
    \\&=&
    \frac1{N^2}\left(\sum_{\ell=1}^\infty|A_{\ell+1}|\Big(\ell^{-1}-(\ell+1)^{-1}\Big)\right).
\end{eqnarray*}
We can estimate this infinite sum by splitting it at some $\ell_0>1$.
Using \eqref{eq_proof_weyl_growth} for $\ell\leq\ell_0$, and the trivial bound $|A_\ell|\leq N^2$ for $\ell>\ell_0$, together with the elementary fact that $\sum_{\ell=\ell_0}^\infty\big(\ell^{-1}-(\ell+1)^{-1}\big)=\ell_0^{-1}$, we conclude that
$$\sum_{\ell=1}^\infty|A_{\ell+1}|\Big(\ell^{-1}-(\ell+1)^{-1}\Big)
\leq
\sum_{\ell=1}^{\ell_0-1}\frac{2 N^2}{\ell g(\log N)^{1+\epsilon}}
+
\frac{N^2}{\ell_0}
\leq
\log(\ell_0)\frac{2 N^2}{g(\log N)^{1+\epsilon}}+
\frac{N^2}{\ell_0}.
$$
Setting $\ell_0=(\log N)^{1+\epsilon}$ and dividing by $N^2$ we obtain
$$\frac1{N^2}\sum_{1\leq m<n\leq N}\min\left(|a(n)-a(m)|^{-1},1\right)
\leq
\frac{1}{(\log N)^{1+\epsilon}}\cdot\left(\frac{2(1+\epsilon)\log\log N}{g}+
1\right)$$
For any $\epsilon'<\epsilon$ and large enough $N$, this bound gets smaller than $\frac{1}{(\log N)^{1+\epsilon'}}$; in other words, the desired estimate \eqref{eq_iedef6} holds.
\end{proof}

The next remark describes a sufficient condition for a sequence to be \ie{} that resembles Weyl's growth condition.
\begin{remark}\label{remark_iegeneral}
    The proof above can be adapted to show that if there exist $\epsilon,g>0$ and $d\in\N$ such that $|a(n)-a(m)|>g$ whenever $m>n+n/(\log n)^{d+\epsilon}$,
    then for any $\epsilon'<\epsilon/d$ and every sufficiently large $N$,
    \begin{equation}\label{eq_iedef6.2}
\frac1{N^2}\sum_{1\leq m<n\leq N}\min\left(|a(n)-a(m)|^{-1/d},1\right)\leq  \frac1{(\log N)^{1+\epsilon'}},
    \end{equation}
    and hence $a(n)$ is $1/d$-\ie.
    
    In particular, if for some $\epsilon>0$ we have $|a(n)-a(m)|>g$ whenever $m>n+n^{1-\epsilon}$, then $a(n)$ is \ie.
\end{remark}

\begin{example}
    For every $\epsilon>0$ the sequence $a(n)=n^\epsilon$ is \ie. 
    This can be quickly verified using \cref{remark_iegeneral}, but it can also be proved directly from the definition. Indeed for a fixed $\delta>0$ and each $m$,
    $$\sum_{n=m+1}^N\min\left(|a(n)-a(m)|^{-\delta},1\right)\leq\sum_{n=1}^Nn^{-\epsilon\delta}\ll N^{1-\epsilon\delta},$$
    so the left hand side of \eqref{eq_iedef} is bounded by $N^{-\epsilon\delta}$.

    More generally, any sufficiently regular sequences that grows faster than any power of $\log$ is $\ie$.
\end{example}

It is not true in general that a \ie{} sequence satisfies Weyl's growth condition, a counterexample is provided by the sequence $a(n)=n-\lfloor\sqrt{n}\rfloor^2$ (note that the first few terms of this sequence are $0,1,2,\,0,1,2,3,4,\,0,1,2,3,4,5,6,\,0,\dots$, and it is not too difficult to show that this sequence is indeed \ie{}; however it does not satisfy Weyl's growth condition since $a(n^2)=0$ for every $n\in\N$).
This example notwithstanding, for sufficiently regular sequences, the following converse of \cref{lemma_weylsgrowthsuffice} holds.
(We won't make use of this fact, but for completeness we present its proof.)
\begin{proposition}
    Let $a:\N\to\R$ be such that the discrete derivative $\Delta a:n\mapsto a(n+1)-a(n)$ is eventually monotone.
    If $a$ is $1$-\ie{}, then there exists $\epsilon>0$ such that $|a(m)-a(n)|>1$ whenever $m>n+n/(\log n)^{1+\epsilon}$.
\end{proposition}

\begin{proof}
Replacing $a$ with $-a$ if needed, we can assume that $a$ is eventually positive.
Since $a$ is \ie, this implies that it is eventually increasing. 
If $\Delta a$ is eventually increasing then $|a(n)-a(m)|\geq g$ for some $g>0$ and all $n,m$ sufficiently large, which implies the desired conclusion, so let us assume that $\Delta a$ is decreasing.

    Let $\epsilon>0$ and suppose that $a(M+M/(\log M)^{1+\epsilon})<a(M)+1$ for infinitely many $M$.
    Then $\Delta a(M)<(\log M)^{1+\epsilon}/M$, and therefore $\Delta a(n)<(\log M)^{1+\epsilon}/M$ for all $n>M$.
    In particular, whenever $n>m>M$ but $n<m+M/(\log M)^{1+\epsilon}$, then $|a_n-a_m|<1$.
    Letting $N=2M$ we have
    \begin{eqnarray*}
        \frac1{N^2}\sum_{1\leq m<n\leq N}\min(|a(n)-a(m)|^{-1},1)
    &\geq&
    \frac1{N^2}\sum_{M\leq m<n<N}1_{n<m+M/(\log M)^{1+\epsilon}}
    \\&\geq&
    \frac1{N^2}\cdot\frac M{(\log M)^{1+\epsilon}}\left(M-\frac M{(\log M)^{1+\epsilon}}\right)
    \\&\geq&
    \frac18\cdot\frac1{(\log M)^{1+\epsilon}}
    \\&\geq&
    \frac18\cdot\frac1{(\log N)^{1+\epsilon}}.
    \end{eqnarray*}
    This precludes \eqref{eq_iedef} from holding with any $\epsilon'<\epsilon$. Since $a$ is \ie, this implies that for some $\epsilon>0$ we must have $|a(m)-a(n)|>1$ whenever $m>n+n/(\log n)^{1+\epsilon}$.
\end{proof}

The final example in this subsection, which was mentioned in \cref{sec_weylgrowth}, complements \cref{example_logcube}.

\begin{example}\label{example_logsquare}
    It is not too difficult to see that the sequence $a(n)=(\log n)^2$ does not satisfy Weyl's growth condition.
    In fact, it is not even $1$-\ie.
    To see this, we first check that for large $N\in\N$, if $n>N/2$ and $m>n-\frac N{8\log N}$, then $a(m)>a(n)-1$. 
    Indeed, we have $m/n>1-\frac N{8n\log N}>1-\frac1{4\log N}$, so $\log m>\log n+\log(1-\frac1{4\log N})>\log n-\frac1{2\log N}$ if $N$ is large enough.
    Therefore
    $$
    a(n)-a(m)
    =
    (\log n)^2-(\log m)^2
    \leq
    \frac{\log n}{\log N}
    \leq
    1,
    $$
    proving the claim.
    Now it follows that $|a(n)-a(m)|\leq1$ whenever $n>N/2$ and $n-\frac N{8\log N}\leq m\leq n$, so
    $$\frac1{N^2}\sum_{1\leq n,m\leq N}\min(\frac1{|a(n)-a(m)|},1)\geq\frac1{N^2}\sum_{1\leq n,m\leq N}1_{|a(n)-a(m)|\leq1}\geq\frac1{16\log N},$$
    and hence $a(n)$ is not $1$-\ie.
\end{example}

\subsection{Revisiting Koksma's theorem}
\label{sec_koksmarevisited}
Note that, while Koksma's \cref{thm_koksmafull} implies \cref{thm_weylrealsbaby}, it does not contain the full theorem of Weyl (\cref{thm_weylreals}) as a special case.
The following theorem, which is a variant of another theorem of Koksma \cite[Satz 4]{Koksma35}, provides a common extension to \cref{thm_koksmafull,thm_weylreals}. 
The rather streamlined proof we present here takes advantage of the notion of \ie{} sequences introduced in \cref{def_ie} and follows the basic steps already present in the proof of \cref{thm_weylintegers}.
\begin{theorem}[Koksma's theorem revisited]\label{thm_betterkoksma}
    Let $a<b$ and, for each $n\in\N$, let $u_n\in C^1([a,b],\R)$.
    Assume that whenever $n\neq m$ the function $x\mapsto\big|u'_n(x)-u_m'(x)\big|$ is monotone, and both sequences $n\mapsto u_n'(a)$ and $n\mapsto u_n'(b)$ are $1$-\ie.
    
    Then for almost every $x\in[a,b]$ the sequence $\big(u_n(x)\big)_{n\in\N}$ is uniformly distributed mod 1.
\end{theorem}

Observe that, in view of \cref{lemma_weylsgrowthsuffice}, \cref{thm_betterkoksma} implies \cref{thm_weylreals} by taking $u_n(x)=a_nx$.

\begin{example}\label{example_koksmafejer}
    Let $u_n(x)=(\log n)^x$.
    Then for almost every $x>2$ the sequence $u_n(x)$ is u.d. mod 1.
    Indeed a standard computation shows that the sequence of functions $(u_n(x))_{n\in\N}$ satisfies the conditions of \cref{thm_betterkoksma}. Note that it does not satisfy the conditions of \cref{thm_koksmafull}.
\end{example}

For the specific $u_n(x)$ in \cref{example_koksmafejer}, one can simply invoke Fej\'er's theorem, which was formulated above after \cref{example_logcube}, and which implies that $u_n(x)$ is uniformly distributed for \emph{every} $x>1$. This conclusion is stronger than that in the example in two ways.
However, the following ``perturbation'' of \cref{example_koksmafejer} cannot be derived from Fej\'er's theorem.

\begin{example}
    Let $u_n(x)=(\log n+\tfrac1n\sin n)^x$.
    Note that $u_n'(x)=(\log n+\tfrac1n\sin n)^x\log(\log n+\tfrac1n\sin n)$.
    An argument similar to the one employed in \cref{example_logcube} shows that $(u_n(x))_{n\in\N}$ satisfies the Weyl growth condition (hence, by \cref{lemma_weylsgrowthsuffice}, it is 1-\ie{}) for each $x>2$.
    Therefore, the sequence of functions $(u_n(x))_{n\in\N}$ satisfies the conditions of \cref{thm_betterkoksma} on $(2,\infty)$ and thus the sequence $(u_n(x))_{n\in\N}$ is u.d. for almost every $x>2$.

    We stress that the sequence $(u_n(x))_{n\in\N}$ does not satisfy the assumptions of \cref{thm_koksmafull} or Fej\'er's theorem (for any fixed $x>2$).
\end{example}

The proof of \cref{thm_betterkoksma} is similar to the proof of its special case, \cref{thm_koksmafull}; the only extra ingredient is the following elementary fact needed to make use of the \ie{} condition.

\begin{lemma}\label{lemma_sublacunary}
Let $z_n$ be a bounded sequence of complex numbers and let $(N_r)_{r\in\N}$ be an increasing sequence of natural numbers such that $N_{r+1}/N_r\to1$ as $r\to\infty$.
If the limit $\lim_{r\to\infty}\frac1{N_r}\sum_{n=1}^{N_r}z_n$ exists, then the limit $\lim_{N\to\infty}\frac1{N}\sum_{n=1}^{N}z_n$ also exists and coincides with the former limit.
\end{lemma}
\begin{proof}
Assuming without loss of generality that $|z_n|\leq1$ for all $n$ we have for all $N<M$ that 
$$\left|\frac1N\sum_{n=1}^Nz_n -\frac1M\sum_{n=1}^Mz_n \right|
=
\left|\sum_{n=1}^Nz_n\frac{M-N}{MN}+\frac1M\sum_{n=N+1}^Mz_n\right|
\leq
2\left(1-\frac NM\right).$$
Therefore, for any $N\in(N_r,N_{r+1})$ we have $\frac1{N}\sum_{n=1}^{N}z_n=\frac1{N_r}\sum_{n=1}^{N_r}z_n+o_{r\to\infty}(1)$.
\end{proof}

\begin{proof}[Proof of \cref{thm_betterkoksma}]
    Fix a non-zero $k\in\Z$ and let $F_N(x):=\frac1N\sum_{n=1}^Ne(k u_n(x))$.
    By Weyl's criterion, if $F_N(x)\to0$ for every non-zero $k\in\Z$, then $\big(u_n(x)\big)_{n\in\N}$ is u.d. mod 1.
    Since $\Z$ is countable, it suffices to show that for a fixed $k$, $F_N(x)\to0$ as $N\to\infty$ for almost every $x\in[a,b]$.
    We have
    $$\|F_N\|^2_{L^2[a,b]}
    =
    \frac1{N^2}\sum_{n,m=1}^N\int_a^b e\Big(k\big(u_n(x)-u_m(x)\big)\Big)\d x.$$
    In view of \cref{lemma_vdCestimate},
    $$\left|\int_a^b e\Big(k\big(u_n(x)-u_m(x)\big)\Big)\d x\right|
    \leq
    \frac1{|k||u_n'(a)-u_m'(a)|}+\frac1{|k||u_n'(b)-u_m'(b)|}.$$
    Using the fact that $|e(t)|=1$ for all $t\in\R$ and the assumption that both sequences $n\mapsto u_n'(a)$ and $n\mapsto u_n'(b)$ are $1$-\ie, it follows that there exists $\epsilon>0$ such that
    $$\|F_N\|^2_{L^2[a,b]}
    \leq
    \frac2{(\log N)^{1+\epsilon}}.$$
    Letting $\delta>0$ small enough that $(1-\delta)(1+\epsilon)>1$ and setting $N_r=e^{r^{1-\delta}}$, we deduce that
    $$\sum_{r=1}^\infty\|F_{N_r}\|^2_{L^2[a,b]}
    \leq
    \sum_{r=1}^\infty\frac2{(\log N_r)^{1+\epsilon}}
    =
    \sum_{r=1}^\infty\frac2{r^{(1-\delta)(1+\epsilon)}}
    <
    \infty,$$
    which implies that for almost every $x\in[a,b]$ the sequence $F_{N_r}(x)\to0$.
    The conclusion now follows from \cref{lemma_sublacunary}.  
\end{proof}

\subsection{Real-analytic versus $C^\infty$ functions}
\label{sec_weylfunctions}

Having established the relevance of \ie{} sequences, in this subsection we discuss smoothness constraints on a function $f$ in order for it to satisfy the following property.
\begin{equation}\label{eq_weylalongfunctions}
    \text{For every \ie{} sequence }a:\N\to\R,\text{ the sequence }\big(a(n)f(x)\big)_{n\in\N}\text{ is u.d. a.e.}
\end{equation}
\cref{thm_weylreals} states that the identity function $f:x\mapsto x$ satisfies \eqref{eq_weylalongfunctions}, and an immediate corollary is that the same is true whenever %$f:\R\to\R$ has the property that 
the pushforward measure\footnote{If $(X,\nu)$ is a measure space and $f:X\to Y$ is a measurable function, then the \emph{pushforward measure} is the measure $f_*\nu$ on $Y$ given by $f_*\nu(B)=\nu(f^{-1}(B))$ for every measurable set $B\subset Y$.} $f_*\Leb$ is absolutely continuous with respect to $\Leb$.
In particular, this is true whenever $f$ is a non-constant (real) analytic function:
\begin{proposition}\label{prop_analyticabscont}
  Let $f:\R\to\R$ be non-constant and real analytic. Then the pushforward measure $f_*\Leb$ is absolutely continuous with respect to $\Leb$ and hence $f$ satisfies \eqref{eq_weylalongfunctions}.
\end{proposition}

\begin{proof}
    Suppose $A\subset\R$ has $\Leb(A)=0$; we need to show that $\Leb(f^{-1}(A))=0$.
    Since $f$ is analytic and non-constant, $\R$ is the union of the countable set $\{x:f'(x)=0\}$ and countable many compact intervals where $f'$ is non-zero. It suffices to prove that for any such interval $I\subset\R$, we have $\Leb(I\cap f^{-1}(A))=0$. 
    
    Let $J=f(I)$. 
    Since $f:I\to J$ is a bijection, it has an inverse $g:J\to I$.
    The change of variable formula now implies that
    \begin{eqnarray*}
        \Leb(I\cap f^{-1}(A))&=&\int_I 1_{f(x)\in A}\d x
    =
    \int_J1_{u\in A}\d u/|f'(g(u))|
    \\&\leq& \max_{x\in I}\frac1{|f'(x)|}\int_J1_{u\in A}\d u
    =
    \max_{x\in I}\frac1{|f'(x)|}\Leb(J\cap A)=0.
    \end{eqnarray*}
\end{proof}

Observe that if $f$ is a constant function, then it does not satisfy \eqref{eq_weylalongfunctions}.
More generally, if there exists a level set $f^{-1}(\{y\}):=\{x\in\R:f(x)=y\}$ with positive measure for some $y\in\R$, then choosing $a$ to be any injective sequence taking values in $\{m\in\N:\|my\|_\T<1/3\}$ we see that there is a positive measure set of $x\in\R$ for which the sequence $\big(a(n)f(x)\big)_{n\in\N}$ is not uniformly distributed mod 1.
One might hope that this is the only obstruction, and any smooth function $f:\R\to\R$ whose level sets have measure zero satisfy $f_*\Leb\ll\Leb$ and hence $f$ satisfies \eqref{eq_weylalongfunctions}. 
However, this is not the case: an example by Yuval Peres answering a question of the authors on Math StackExchange \cite{Peres23} produces a function $f\in C^\infty(\R)$  with $\Leb(f^{-1}(y))=0$ for all $y\in\R$ and such that $\Leb(f^{-1}(\Lambda))>0$ where $\Lambda$ is a Cantor set on $\R$ with $\Leb(\Lambda)=0$. 

In view of Peres' example and \cref{prop_analyticabscont}, in this paper all functions are assumed to be real analytic.

\subsection{A multidimensional Weyl theorem along polynomial curves}
\label{sec_surveylast}
We conclude this section by demonstrating how a small variation of Weyl's proof already suffices to establish \cref{thm_muldidimpolWeyl2} which involves polynomial curves and uniform distribution in the multidimensional setting.

The presence of the scalar product in Weyl's criterion for multidimensional sequences (\cref{thm_weylcriterion}) leads us to consider the following notion for $k$-tuples of sequences.
\begin{definition}[Jointly \ie{} sequences]\label{def_jointie}
Given sequences $a_1,\dots,a_k:\N\to\R$, we say that they are \emph{jointly \ie} if for every non-zero $v=(v_1,\dots,v_k)\in\R^k$ the linear combination $v_1a_1+\cdots+v_ka_k$ is \ie.
\end{definition}

For the proof of \cref{thm_muldidimpolWeyl2}, and its extension \cref{thm_polKoksmaWeyl}, we will require an extension of van der Corput's \cref{lemma_vdCestimate} involving higher derivatives.
\begin{theorem}
    [{\cite[Lemma 2.2, page 16]{Kuipers_Niederreiter74}},{\cite[Proposition 2, Chapter VIII]{Stein93}}]\label{lemma_highordervdC}
Let $a<b$ be real, let $d\in\N$, $d\geq2$ and let $f\in C^{d}[a,b]$ be real valued. 
Then
$$\left|\int_a^be\big(f(x)\big)\d x\right|\leq C_d\sup\Big\{\big|f^{(d)}(x)\big|^{-1/d}:x\in[a,b]\Big\},$$
where the constant $C_d$ depends only on $d$.
\end{theorem}
Strictly speaking, \cref{lemma_highordervdC} does not contain \cref{lemma_vdCestimate} as a special case, since it requires $d\geq2$. In the $d=1$ case, one needs the additional assumption that $f'$ is monotone; this assumption is unnecessary in \cref{lemma_highordervdC} because the inequality holds trivially if $f^{(d)}$ has a zero, and otherwise there is a bounded number of sub-intervals of $[a,b]$ where $f'$ is monotone.

We are now ready to prove \cref{thm_muldidimpolWeyl2} which we restate here for the convenience of the reader.
\begin{named}{\cref{thm_muldidimpolWeyl2}}{}
    Let $k\in\N$, let $p_1,\cdots, p_k\in\R[x]$ be non-constant polynomials and let $a_1,\dots,a_k:\N\to\R$ be jointly \ie.
    Then for Lebesgue-a.e.~$x\in\R$, the sequence $\big(a_1(n)p_1(x),\dots,a_k(n)p_k(x)\big)_{n\in\N}$ is uniformly distributed on $\T^k$.
\end{named}

\begin{proof}
    Let $k\in\N$, let $p_1,\dots,p_k\in\R[x]$ be non-constant polynomials and let $a_1,\dots,a_k:\N\to\R$ be jointly \ie{} sequences. Our goal is to show that for almost every $x\in\R$ the sequence $\big(a_1(n)p_1(x),\dots,a_k(n)p_k(x)\big)_{n\in\N}$ is uniformly distributed on $\T^k$.

    Using \cref{thm_weylcriterion} if suffices to show that for each fixed non-zero vector $v=(v_1,\dots,v_k)\in\Z^k$ the sequence of functions $F_N:\R\to\C$ defined by
    $$F_N(x)=\frac1N\sum_{n=1}^Ne\big(v_1a_1(n)p_1(x)+\cdots+v_ka_k(n)p_k(x)\big)$$
    tend to zero almost everywhere on $\R$ as $N\to\infty$.
    Since $\R$ is the countable union of (not necessarily disjoint) compact intervals, it suffices to establish almost everywhere convergence to $0$ on each such interval. 

    Let $d\in\N$ be the maximal degree among the polynomials $\{p_i:v_i\neq0\}$.
    Since $v$ is non-zero, $d$ is well defined.
    For a fixed $n\in\N$, the function $\phi_n:x\mapsto v_1a_1(n)p_1(x)+\cdots+v_ka_k(n)p_k(x)$ is a polynomial of degree at most $d$, and the coefficient $b(n)$ of the term $x^d$ in $\phi_n$ can be expressed as a non-zero linear combination of the sequences $a_1(n),\dots,a_k(n)$.
    Since these sequences are jointly \ie, the sequence $b(n)$ is \ie.
    The $d$-th derivative of $\phi_n$ is the constant function $d!b(n)$, and in particular it is also a \ie{} sequence.
    
    Let $I\subset\R$ be an arbitrary compact interval.
    To show that $F_N\to0$ almost everywhere on $I$, we consider the $L^2$ norm
    $$\big\|F_N\big\|_{L^2(I)}^2
    =
    \int_I\left|\frac1N\sum_{n=1}^Ne\big(\phi_n(x)\big)\right|^2\d x
    =
    \frac1{N^2}\sum_{n,m=1}^N\int_I
    e\Big(\phi_n(x)-\phi_m(x)\Big)\d x.$$
Using \cref{lemma_highordervdC}, and \eqref{eq_iedef} for the sequence $b(n)$, we can estimate
$$\big\|F_N\big\|_{L^2(I)}^2
\leq
C_d\frac1{N^2}\sum_{n,m=1}^N\min\Big(\big|b(n)-b(m)\big|^{-1/d},|I|\Big)\leq\frac1{(\log N)^{1+\epsilon}}$$
for some $\epsilon>0$ and every sufficiently large $N$ (depending only on $d$ and $|I|$).

To finish the proof, we use the same argument as in the proof of \cref{thm_betterkoksma}.
    Let $\epsilon'>0$ be sufficiently small so that $(1-\epsilon')(1+\epsilon)>1$. 
    Setting $N_r:=e^{r^{1-\epsilon'}}$, we deduce that
    $$\sum_{r=1}^\infty\|F_{N_r}\|^2_{L^2(I)}
    \leq
    \sum_{r=1}^\infty\frac1{(\log N_r)^{1+\epsilon}}
    =
    \sum_{r=1}^\infty\frac1{r^{(1-\epsilon')(1+\epsilon)}}
    <
    \infty,$$
    which implies that the function $x\mapsto\sum_{r=1}^\infty\big|F_{N_r}(x)\big|^2$ has a finite integral and hence is finite for almost every $x\in I$.
    In particular, for almost every $x\in I$ the sequence $F_{N_r}(x)\to0$.
    The conclusion now follows from \cref{lemma_sublacunary}.
\end{proof}

\begin{remark}
  The proof of \cref{thm_muldidimpolWeyl2} can be adapted to handle the case where the sequences $a_1,\dots,a_k$ are \ie{} (but not necessarily jointly \ie) and the polynomials $p_1,\dots,p_k$ have distinct degrees.
However, to prove \cref{thm_muldidimpolWeyl1} in full generality, we need a different method.
To clarify this point, consider the sequence $\big((n+\log n)x^2,n(x^2+x)\big)$, corresponding to 
$$a_1(n)=n+\log n,\qquad a_2(n)=n,\qquad p_1(x)=x^2,\qquad p_2(x)=x^2+x.$$
When $v=(1,-1)$, in the proof above, we get $\phi_n(x)=(\log n)x^2-nx$. 
Although this sequence \emph{is} uniformly distributed mod 1 for almost all $x\in\R$, this cannot be ascertained by differentiating $\phi_n$ twice, because the sequence $(\log n)x^2$, which has the same second derivative, is not uniformly distributed for any $x\in\R$. 

\end{remark}

In the next section we present a proof of \cref{thm_mainmultidimweyl1} which contains \cref{thm_muldidimpolWeyl1} as a special case.

\section{Multidimensional Weyl theorem along analytic curves}
\label{sec_proofofmainthm}

The goal of this section is to prove \cref{thm_mainmultidimweyl3}.

\subsection{The case of linearly independent functions}

First we deal with the case of \cref{thm_mainmultidimweyl3} which involves linearly independent functions.
In this case we obtain the following ``if and only if'' statement which is of interest on its own.

\begin{theorem}\label{thm_mainmultidimweyl1}
        Let $k\in\N$ and let $f_1,\dots,f_k:\R\to\R$ be analytic functions. Then the following are equivalent: 
        \begin{enumerate}
            \item  The set $\{1,f_1,\dots,f_k\}$ is linearly independent over $\R$.
            \item For every \ie{} sequences $a_1,\dots,a_k:\N\to\R$, there exists a full measure set of $x\in\R$ for which the sequence $\big(a_1(n)f_1(x),\dots,a_k(n)f_k(x)\big)_{n\in\N}$ is uniformly distributed on $\T^k$.
        \end{enumerate}
    \end{theorem}

        \begin{proof}

        First we prove the implication (2)$\Rightarrow$(1).
        Suppose $c_1,\dots,c_k\in\R$ are not all zero but the linear combination $c_1f_1+\cdots+c_kf_k$ is a constant function with value $\alpha\in\R$.
        Then, taking 
        $$a_i(n)=\begin{cases}
            n&\text{ if }c_i=0\\
            c_in&\text{ if }c_i\neq0\text{ and }\alpha=0\\
            \tfrac{c_i}\alpha n&\text{ if }c_i\neq0\text{ and }\alpha\neq0,
        \end{cases}$$
        which are all \ie{} sequences, we claim that the sequence $\big(a_1(n)f_1(x),\dots,a_k(n)f_k(x)\big)_{n\in\N}$ is not uniformly distributed on $\T^k$ for every $x\in\R$.

        To verify the claim, let $v_i=1_{c_i\neq0}$ for each $i=1,\dots,k$ and note that
        $$v_1a_1(n)f_1(x)+\cdots+v_ka_k(n)f_k(x)=\begin{cases}
            0&\text{ if }\alpha=0\\
            n&\text{otherwise},
        \end{cases}$$
        for every $x\in\R$ and $n\in\N$.
        It follows that
        $$\frac1N\sum_{n=1}^Ne\Big(v_1a_1(n)f_1(x)+\cdots+v_ka_k(n)f_k(x)\Big)=1\neq0,$$
        so the Weyl criterion (\cref{thm_weylcriterion}) implies that $\big(a_1(n)f_1(x),\dots,a_k(n)f_k(x)\big)_{n\in\N}$ is not uniformly distributed on $\T^k$, verifying the claim and finishing the proof of the implication (2)$\Rightarrow$(1).

        Next we prove the converse implication (1)$\Rightarrow$(2).
            Let $k\in\N$ and $f_1,\dots,f_k:\R\to\R$ be analytic functions such that $\{1,f_1,\dots,f_k\}$ are linearly independent over $\R$ and fix \ie{} sequences $a_1,\dots,a_k$.
            By the Weyl criterion, for any fixed $x\in\R$ the sequence $(a_1(n)f_1(x),\dots,a_k(n)f_k(x))_{n\in\N}$ is u.d. on $\T^k$ if and only if for every $v=(v_1,\dots,v_k)\in\Z^k\setminus\{\vec0\}$,
            \begin{equation}
                \label{eq_proofLeveque1}
                \lim_{N\to\infty}\frac1N\sum_{n=1}^Ne\big(v_1a_1(n)f_1(x)+\cdots+v_ka_k(n)f_k(x)\big)=0.
            \end{equation}
            Since $\Z^k$ is countable, it suffices to show that for any given $v\in\Z^k\setminus\{\vec0\}$, \eqref{eq_proofLeveque1} holds for almost every $x\in\R$.
            Fix from now on such $v$ and denote by $\phi_n(x)=v_1a_1(n)f_1(x)+\cdots+v_ka_k(n)f_k(x)$.
            
            We seek to prove that the sequence $\frac1N\sum_{n=1}^Ne\big(\phi_n(x)\big)$ converges to $0$ as $N\to\infty$ almost everywhere in $\R$.
            Since $\R$ is a countable union of (not necessarily disjoint) compact intervals, it suffices to show that for each fixed compact interval $I\subset\R$, the functions $F_N:I\to\C$ given by $F_N(x)=\frac1N\sum_{n=1}^Ne\big(\phi_n(x)\big)$ converge to $0$ almost everywhere as $N\to\infty$.
            We fix such an interval $I$ from now on.
              
            We claim that there exists $\epsilon>0$ such that
            \begin{equation}\label{eq_proofLeveque2}
                \|F_N\|_{L^2(I)}^2\ll\frac1{\log(N)^{1+\epsilon}}.
            \end{equation}
            Assuming for now that the claim is true, for each $r\in\N$ let $N_r\in\N$ be the smallest integer satisfying $\log N_r>r^{1-\epsilon/2}$.
            Note that $N_{r+1}/N_r\to1$ as $r\to\infty$ and $\log(N_r)^{1+\epsilon}\geq r^{1+\epsilon/4}$.
            It follows from \eqref{eq_proofLeveque2} and the monotone convergence theorem that $S(x):=\sum_{r=1}^\infty|F_{N_r}(x)|^2$ has $\|S\|_{L^1(I)}<\infty$ and hence $S(x)<\infty$ for almost every $x\in I$.
            This implies that $F_{N_r}(x)\to0$ for almost every $x\in I$, and in view of \cref{lemma_sublacunary}, it then implies that $F_N(x)\to0$ for almost every $x\in I$. 
            This is what we wanted to show, so we are left to establish the claim \eqref{eq_proofLeveque2}.

            Using \cref{thm_wright}, let $C,\delta>0$ be such that \eqref{eq_wrightestimate} holds.
            For each $n,m\in\N$ let $\lambda_{n,m}\in\R^k$ be the vector
            $$\lambda_{n,m}=\Big(v_1\big(a_1(n)-a_1(m)\big),\dots,v_k\big(a_k(n)-a_k(m)\big)\Big).$$
            Since $v$ is non-zero, there exists $j\in\{1,\dots,k\}$ such that $v_j\neq0$.
            It follows that for every $n,m\in\N$, $\|\lambda_{n,m}\|\geq|v_j|\cdot\big|a_j(n)-a_j(m)\big|$ and hence 
            \begin{equation}\label{eq_proofLeveque3}
                \|\lambda_{n,m}\|^{-\delta}\leq|v_j|^{-\delta}\cdot\big|a_j(n)-a_j(m)\big|^{-\delta}.
            \end{equation}
            
            We rewrite the left hand side of \eqref{eq_proofLeveque2} as
            \begin{equation}
                \label{eq_proofLeveque4.1}
                \|F_N\|^2_{L^2(I)}
            =
            \frac1{N^2}\sum_{1\leq n,m\leq N}\int_Ie\left(\sum_{i=1}^k\lambda_{n,m;i}f_i(x)\right)\d x.
            \end{equation}
            For each $n,m$, the integrals in the right-hand side of \eqref{eq_proofLeveque4.1} are bounded by $|I|$, so combining \eqref{eq_wrightestimate} with \eqref{eq_proofLeveque3} and then using the fact that $a_j$ is \ie{} via \eqref{eq_iedef}, we conclude that there exists $\epsilon>0$ such that

             $$
            \|F_N\|^2_{L^2(I)}\leq C|v_j|^{-\delta}\frac1{N^2}\sum_{1\leq n,m\leq N}\min\Big(|I|,\big|a_j(n)-a_j(m)\big|^{-\delta}\Big)
            \ll
            \frac 1{(\log N)^{1+\epsilon}}
            $$
            establishing \eqref{eq_proofLeveque2} and finishing the proof.            
        \end{proof}

\subsection{The case of jointly \ie{} sequences}
In this subsection we prove the case of \cref{thm_mainmultidimweyl3} where no independence assumption on the $f_i$ is made, but we require the sequences $a_1,\dots,a_k$ to be jointly \ie; this is the content of the following theorem.

    \begin{theorem}
    \label{thm_mainmultidimweyl2}
    Let $k\in\N$, let $f_1,\dots,f_k:\R\to\R$ be non-constant analytic functions.
        Then whenever $a_1,\dots,a_k:\N\to\R$ are jointly \ie{}, for Lebesgue-a.e.~$x\in\R$, the sequence $\big(a_1(n)f_1(x),\dots,a_k(n)f_k(x)\big)_{n\in\N}$ is uniformly distributed on $\T^k$.
\end{theorem}

\begin{proof}
     
    We begin as in the proof of \cref{thm_mainmultidimweyl1}: after invoking Weyl's criterion \cref{thm_weylcriterion} it suffices to prove that for a fixed $v\in\Z^k\setminus\{\vec0\}$ and a fixed compact interval $I\subset\R$, the functions $F_N:I\to\C$ given by
    $$F_N(x):=\frac1N\sum_{n=1}^Ne\big(v_1a_1(n)f_1(x)+\cdots+v_ka_k(n)f_k(x)\big)$$
    converge to $0$ almost everywhere on $I$.
    Arguing still exactly as in the proof \cref{thm_mainmultidimweyl1}, we use \cref{lemma_sublacunary} to further reduce matters to finding $\epsilon>0$ such that
            \begin{equation}\label{eq_proofLeveque21}
                \|F_N\|_{L^2(I)}^2\ll\frac1{\log(N)^{1+\epsilon}}.
            \end{equation}
    The rest of the proof is dedicated to establishing \eqref{eq_proofLeveque21}.
    
    For ease of notation set $f_0\equiv1$. 
    Reordering the $f_i$ we may assume that there exists $d\in\{1,\dots,k\}$ such that the functions $\{f_0,\dots,f_d\}$ are linearly independent, and for each $i\in\{d+1,\dots,k\}$ we have 
    $$f_i=u_{i,0}f_0+\cdots+u_{i,d}f_d$$ for some real coefficients $u_{i,j}\in\R$.
    We also set $u_{i,j}=1_{i=j}$ for $i\in\{1,\dots,d\}$. 
    Let
    $$\phi_n(x)
    =
    \sum_{i=1}^kv_ia_i(n)f_i(x)
    =
    \sum_{i=1}^kv_ia_i(n)\sum_{j=0}^du_{i,j}f_j(x)
    =
    \sum_{j=0}^db_j(n)f_j(x),$$
    where for each $j\in\{0,\dots,d\}$ we let $b_j(n)=v_1a_1(n)u_{1,j}+\cdots+v_ka_k(n)u_{k,j}$.
    We claim that there exists $j\in\{1,\dots,d\}$ for which $b_{j}$ is \ie.
    Since $v\neq\vec0$, there exists some $i\in\{1,\dots,k\}$ such that $v_i\neq0$, and because $f_i$ is not constant, there exists some $j\in\{1,\dots,d\}$ such that $u_{i,j}\neq0$.
    Since the sequences $a_1,\dots,a_k$ are jointly \ie{} and $v_iu_{i,j}\neq0$, the sequence $b_{j}$ is \ie.
    We now use \cref{thm_wright} with $f_1,\dots,f_d$ to find $C,\delta>0$ such that for every $\lambda\in\R^d$
    \begin{equation}\label{eq_proof_thm_mainmultidimweyl2.1}
        \left|\int_Ie\big(\lambda_1f_1(x)+\cdots+\lambda_df_d(x)\big)\d x\right|\leq C\|\lambda\|^{-\delta}.
    \end{equation}
    For each $n,m\in\N$ let $\lambda_{n,m}\in\R^d$ be the vector
    $$\lambda_{n,m}=\big(b_1(n)-b_1(m),\dots,b_d(n)-b_d(m)\big).$$
    Note that $\|\lambda_{n,m}\|\geq\big|b_{j}(n)-b_j(m)\big|$ and hence
    \begin{equation}\label{eq_lambdamnestimate2}
        \|\lambda_{n,m}\|^{-\delta}\leq\big|b_j(n)-b_j(m)\big|^{-\delta}.
    \end{equation}

     Expanding the square we rewrite the left hand side of \eqref{eq_proofLeveque21} as
            \begin{eqnarray*}
                \|F_N\|^2_{L^2(I)}
            &=&\left\|\frac1N\sum_{n=1}^Ne\big(\phi_n(x)\big)\right\|^2_{L^2(I)}
            =
            \frac1{N^2}\sum_{1\leq n,m\leq N}\int_Ie\big(\phi_n(x)-\phi_m(x)\big)\d x
            \\&=&
            \frac1{N^2}\sum_{1\leq n,m\leq N}\int_Ie\left(\sum_{i=1}^d\big(b_i(n)-b_i(m)\big)f_i(x)\right)\d x
            \\&=&
            \frac1{N^2}\sum_{1\leq n,m\leq N}\int_Ie\left(\sum_{i=d}^k\lambda_{n,m;i}f_i(x)\right)\d x.
            \end{eqnarray*}
            The integral above is always bounded by $|I|$, so combining \eqref{eq_proof_thm_mainmultidimweyl2.1} with \eqref{eq_lambdamnestimate2} and then using the fact that $b_j$ is \ie{} via \eqref{eq_iedef}, we conclude that there exists $\epsilon>0$ such that

             $$
            \|F_N\|^2_{L^2(I)}\leq C\frac1{N^2}\sum_{1\leq n,m\leq N}\min\Big(|I|,\big|b_j(n)-b_j(m)\big|^{-\delta}\Big)
            \ll
            \frac 1{(\log N)^{1+\epsilon}}
            $$
            establishing \eqref{eq_proofLeveque21} and finishing the proof.  
\end{proof}

\section{Multidimensional extensions of Koksma's theorem}\label{sec_koksmaweyl}

In this section we explore multidimensional extensions of Koksma's \cref{thm_koksmafull}. 
In subsection \ref{section_koksmafirst} we present a proof of \cref{thm_polKoksmaWeyl}, which combines \cref{thm_koksmafull} with the multidimensional \cref{thm_muldidimpolWeyl1,thm_muldidimpolWeyl2}.
Then in subsection \ref{sec_multikoksma} we present some observations and conjectures that are motivated by \cref{thm_mainmultidimweyl3}.

\subsection{Proof of \cref{thm_polKoksmaWeyl}}\label{section_koksmafirst}
In this subsection we prove \cref{thm_polKoksmaWeyl} whose statement we now restate for convenience of the reader.
\begin{named}{\cref{thm_polKoksmaWeyl}}{}
        Let $g:\R\to(1,\infty)$ be a non-constant analytic function and let $b:\N\to\R$ tend to $\infty$ and be \ie{}.
        Let $k\in\N$, let $p_1,\cdots, p_k\in\R[x]$ be non-constant polynomials and let $a_1,\dots,a_k:\N\to\R$ be \ie{}.
        Assume that either
    \begin{enumerate}
        \item\label{part1_thm_muldidimanalWeyl} The sequences $a_1,\dots,a_k$ are jointly \ie{}, 
        
        \item[or] 
        \item\label{part2_thm_muldidimanalWeyl} The polynomials $\{1,p_1,\dots,p_k\}$ are linearly independent.
    \end{enumerate}
        Then, for almost every $x\in\R$, the sequence $\big(g(x)^{b(n)},a_1(n)p_1(x),\dots,a_k(n)p_k(x)\big)$ is u.d. in $\T^{k+1}$.
\end{named}

\begin{proof}
    In view of Weyl's criterion (\cref{thm_weylcriterion}), it suffices to show that for any non-zero $v\in\Z^{k+1}$, the sequence of functions $F_N:\R\to\C$ given by
$$F_N(x):=\frac1N\sum_{n=1}^Ne\left(v_0g(x)^{b(n)}+\sum_{i=1}^kv_ia_i(n)p_i(x)\right)$$
tends to zero almost everywhere.
Since $g$ is analytic, the set $\{x\in\R:g'(x)=0\}$ is countable and closed. 
Therefore, $\R$ can be covered, up to a countable set, by countably many compact intervals where $g'(x)$ does not vanish.
Hence we may restrict attention to a single compact interval $I$ where $g'(x)$ does not vanish, and our goal is to show that the (restriction to $I$ of the) sequence $F_N(x)$ tends to zero almost everywhere on $I$.

Whenever $v_0=0$, the conclusion that $F_N(x)\to0$ for a.e. $x\in I$ follows from the exact same argument that was used to prove \cref{thm_mainmultidimweyl3}.
Suppose next that $v_0\neq0$.

We begin by computing the norms
$$\|F_N\|^2_{L^2(I)}=\frac1{N^2}\sum_{n,m=1}^N\int_Ie\left(v_0\big(g(x)^{b(n)}-g(x)^{b(m)}\big)+\sum_{i=1}^kv_ip_i(x)\big(a_i(n)-a_i(m)\big)\right)\d x$$
We choose $d\in\N$ larger than the degree of all the $p_i$ and use \cref{lemma_highordervdC} to deduce that
\begin{equation}\label{eq_proof_kosksmaweyl1}
    \|F_N\|^2_{L^2(I)}
    \leq
    \frac1{N^2}\sum_{n,m=1}^N \min\bigg(C_d\cdot\sup\Big\{\Big|f_{n,m}^{(d)}(x)\Big|^{-1/d}:x\in I\Big\},1\bigg),
\end{equation}
where $f_{n,m}(x)=v_0\big(g(x)^{b(n)}-g(x)^{b(m)}\big)$.
Denote by $\psi(x,b)=g(x)^b$ for any $x,b\in\R$, and fix a bound $b_0>0$ to be determined later.
For $n,m\in\N$ such that $b(n),b(m)>b_0$, using the mean value theorem we can estimate
\begin{equation}\label{eq_proof_kosksmaweyl3}
  \forall x\in I,\qquad  \Big|f_{n,m}^{(d)}(x)\Big|
=
|v_0|\cdot\left|\frac{\partial^d\psi}{\partial x^d}\big(x,b(n)\big)-\frac{\partial^d\psi}{\partial x^d}\big(x,b(m)\big)\right|
\geq
C\big|b(n)-b(m)\big|,
\end{equation}
where $C=C(b_0)=|v_0|\cdot\inf\left\{\left|\frac{\partial}{\partial b}\frac{\partial^d\psi}{\partial x^d}(x,b)\right|:x\in I, |b|>b_0\right\}$.

We now show that for some $b_0>0$, the value of $C$ is positive.
A routine induction (or direct computation) shows that for each fixed $x\in I$ and $d\in\N$, the map
$$\phi_x:b\mapsto \frac1{\psi(x,b)}\cdot \frac{\partial^d\psi}{\partial x^d}(x,b)$$%\frac{d^d}{dx^d}\Big(g(x)^b\Big)(y)$$
is a polynomial of degree $d$. 
Moreover, the coefficients of $\phi_x$ depend continuously on  $x$ (in fact, they are polynomials on $g(x)$ and its derivatives) and the leading coefficient is $(g'(x))^d$. 
We further compute
$$\frac{\partial}{\partial b}\frac{\partial^d\psi}{\partial x^d}(x,b)
=
\frac{\partial}{\partial b}\big( g(x)^b\cdot \phi_x(b)\big)
=
g(x)^b\Big(\phi'_x(b)-\phi_x(b)\log\big(g(x)\big)\Big).$$
The map $P_x:b\mapsto\phi'_x(b)-\phi_x(b)\log\big(g(x)\big)$ is again a polynomial whose coefficients depend continuously on $x$ and whose leading coefficient is $(g'(x))^d\log\big(g(x)\big)$.
Since $I$ is compact, $g'(x)\neq0$ on $I$ and $g$ takes values in $(1,\infty)$, it follows that if $b_0$ is sufficiently large (depending on $I,g$ and $d$) then whenever $|b|>b_0$ we have $P_x(b)>1$ for every $x\in I$.
For this value of $b_0$, the value of $C$ in \eqref{eq_proof_kosksmaweyl3} is positive.

Combining the two estimates \eqref{eq_proof_kosksmaweyl1} and \eqref{eq_proof_kosksmaweyl3} with the fact that $b$ is $1/d$-\ie, we obtain
\begin{eqnarray*}
    \|F_N\|^2_{L^2(I)}
    &\ll&
    \frac1{N^2}\sum_{n,m=1}^N\min\big(|b(n)-b(m)|^{-1/d},1\big)+o(1)
    \\&\ll&
\frac1{(\log N)^{1+\epsilon'}}
\end{eqnarray*}
for some $\epsilon'>0$.
To finish the proof, we use yet again the same argument as in the proof of \cref{thm_betterkoksma}.
    Let $\epsilon''>0$ be sufficiently small so that $(1-\epsilon'')(1+\epsilon')>1$. 
    Setting $N_r:=e^{r^{1-\epsilon''}}$, we deduce that
    $$\sum_{r=1}^\infty\|F_{N_r}\|^2_{L^2(I)}
    \leq
    \sum_{r=1}^\infty\frac1{(\log N_r)^{1+\epsilon'}}
    =
    \sum_{r=1}^\infty\frac1{r^{(1-\epsilon'')(1+\epsilon')}}
    <
    \infty,$$
    which implies that the function $x\mapsto\sum_{r=1}^\infty\big|F_{N_r}(x)\big|^2$ has a finite integral and hence is finite for almost every $x\in I$.
    In particular, for almost every $x\in I$ the sequence $F_{N_r}(x)\to0$.
    The conclusion now follows from \cref{lemma_sublacunary}.
\end{proof}

Note that the proof uses crucially the fact that only the first term was not a polynomial of bounded degree in $x$, so when applying enough derivatives all terms vanish except for the first one.
The same method faces more obstacles when trying to prove that, for example, $(x^n,x^{2n})$ is uniformly distributed in $\T^2$ for almost every $x>1$; we are unable to either prove or disprove this statement (see \cref{conj_twodimKoksma} below for a more general statement).
In short, in such cases we cannot rule out the possibility that the supremum in \eqref{eq_proof_kosksmaweyl1} is large, or even infinite, for many choices of $n,m$.

\subsection{Multidimensional variants of \cref{thm_betterkoksma}}
\label{sec_multikoksma}

While \cref{thm_curvesformulation} gives a satisfactory extension of Weyl's theorem to curves in higher dimensions, similar multidimensional extensions of Koksma's \cref{thm_betterkoksma} are significantly more challenging.
\cref{thm_polKoksmaWeyl} proved in the previous subsection, while containing the one dimensional \cref{thm_betterkoksma} as a special case, is suggestive of several further amplifications.
The goal of this subsection is to discuss some conjectures in this direction.

We start the discussion with the following simple proposition, which for expository reasons we formulate only in two dimensions.
\begin{proposition}\label{prop_2dimKoksma}
Let $(u_n)_{n\in\N}$ and $(v_n)_{n\in\N}$ satisfy the conditions of \cref{thm_betterkoksma} on the intervals $I$ and $J$ respectively.
Then there is a full measure set of pairs $(x,y)\in I\times J$ for which the sequence $\big(u_n(x),v_n(y)\big)_{n\in\N}$ is uniformly distributed in $\T^2$.
\end{proposition}

\begin{proof}
    In view of Weyl's criterion (\cref{thm_weylcriterion}), it suffices to show that for each $(a,b)\in\Z^2\setminus\{(0,0)\}$ there is a full measure set of pairs $(x,y)\in I\times J$ for which 
    \begin{equation}\label{eq_proof_prop_2dimKoksma1}
        \lim_{N\to\infty}\frac1N\sum_{n=1}^Ne\big(au_n(x)+bv_n(y)\big)=0.
    \end{equation}
    If $a\neq0$, then we will show that for every $y\in J$ there is a full measure set of $x\in I$ for which \eqref{eq_proof_prop_2dimKoksma1} holds.
    Indeed, in this case the sequence $w_n(x):=au_n(x)+bv_n(y)$ satisfies the conditions of \cref{thm_betterkoksma}, which then implies that for almost every $x\in I$ the sequence $w_n(x)$ is uniformly distributed modulo one.
    Weyl's criterion then implies for each such $x$, \eqref{eq_proof_prop_2dimKoksma1} holds.

    If $a=0$ then $b\neq0$. In this case we can reverse the roles of the two coordinates above and deduce that for every $x\in I$ there exists a full measure set of $y\in J$ for which \eqref{eq_proof_prop_2dimKoksma1} holds.
\end{proof}

One may hope to find a ``curve'' version of the above proposition that extends Koksma's \cref{thm_betterkoksma} in the same way as \cref{thm_mainmultidimweyl3} extends Weyl's \cref{thm_weylrealsbaby}. However one needs to heed the following simple (counter)example.

\begin{example}\label{example_badforKoskma}
     Take $u_n(x)=nx$, $v_n(x)=n\sqrt{x}$, $f_1(x)=x$, $f_2(x)=x^2$. Then $u_1$ and $u_2$ both satisfy the conditions of \cref{thm_betterkoksma} and $\{1,f_1,f_2\}$ are linearly independent; however the pair $(u_n(f_1(x)),v_n(f_2(x)))=(nx,nx)$ is not uniformly distributed in $\T^2$ for any $x$.
\end{example}

We propose two possible ways to avoid the pitfall illustrated by \cref{example_badforKoskma}.

\begin{conjecture}\label{conj_twodimKoksma}
   Let $(u_n)_{n\in\N}$ satisfy the conditions of \cref{thm_betterkoksma}. 
   Let $f_1,\dots,f_k:\R\to\R$ be analytic functions such that $\{1,f_1,\dots,f_k\}$ are linearly independent.
Then for almost every $x\in\R$ the sequence $$\big(u_n(f_1(x)),\dots,u_n(f_k(x))\big)_{n\in\N}.$$ 
is uniformly distributed in $\T^k$. 
\end{conjecture}

\begin{conjecture}
   Let $I\subset\R$ be a compact interval, $k\in\N$ and, for each $i=1,\dots,k$, let $(u_{i;n})_{n\in\N}$ be a sequence of functions in $C^1(I,\R)$. Assume that the derivatives $u_{i,n}'$ are monotone for every $i$ and $n$.
   Let $f_1,\dots,f_k:\R\to\R$ be analytic functions.

   Assume that for every $a\in\R$ the sequences $u_{1,n}'(a),\dots,u_{k,n}'(a)$ are jointly \ie.
Then for almost every $x\in\R$ the sequence $$\big(u_{1,n}(f_1(x)),\dots,u_{k,n}(f_k(x))\big)_{n\in\N}.$$ 
is uniformly distributed in $\T^k$. 
\end{conjecture}

Note that either of the above conjectures implies that the sequence $(x^n,x^{2n})_{n=1}^\infty$ is uniformly distributed on $\T^2$ for almost every $x>1$.

Next we focus our attention on the special case when the sequences $u_n(x)$ take either the form $a(n)f(x)$ or the form $g(x)^{b(n)}$.
This was the setting of \cref{question_KoksmaWeyl2k}. We now conjecture a possible answer, which extends the scope of  \cref{conj_koksma} formulated in the introduction.

\begin{conjecture}\label{conjecture_KoksmaWeyl2k}
    Let $g_1,\dots,g_k:\R\to(1,\infty)$ and $f_1,\dots,f_k:\R\to\R$ be analytic functions.
    Let $a_1,\dots,a_k,b_1,\dots,b_k:\N\to\R$ be sequences such that $a_1,\dots,a_k$ are jointly \ie{} and $b_1,\dots,b_k$ are jointly \ie{}.

    Assume that both sets $\{1,g_1,\dots,g_k\}$ and $\{1,f_1,\dots,f_k\}$ are linearly independent over $\R$. Then for almost every $x\in\R$ the sequence
    $$\big(g_1(x)^{b_1(n)},\dots,g_k(x)^{b_k(n)},a_1(n)f_1(x),\dots,a_k(n)f_k(x)\big)$$
    is uniformly distributed on $\T^{2k}$.
\end{conjecture}

Using a ``change of variable trick'' we can prove the following result, corresponding to a special case of \cref{conjecture_KoksmaWeyl2k}. Note that, in comparison with \cref{thm_polKoksmaWeyl}, we have only one $f$, but it can be any analytic function, not necessarily a polynomial. 
\begin{named}{\cref{thm_weyl_koksmafusion}}{}
        Let $f,g:\R\to\R$ be non-constant analytic functions such that $g(\R)\subset(1,\infty)$ and let $a,b:\N\to\R$ be \ie{} sequences. 
        Then for almost every $x\in\R$ the sequence $\big(g(x)^{b(n)},a(n)f(x)\big)$ is u.d. on $\T^2$.
\end{named}
\begin{proof}
    Let $I_0\subset\R$ be a compact interval where $f'$ does not vanish and let $I:=f(I_0)$ and $h=g\circ f^{-1}$. 
    Note that $f:I_0\to I$ is a bijection, so $f^{-1}$ and hence $h$ is well-defined on $I$.
    Applying \cref{thm_polKoksmaWeyl} we conclude that for almost every $x\in I$, the sequence $\big(h(x)^{b(n)},a(n)x\big)$ is u.d. on $\T^2$.
    In view of \cref{prop_analyticabscont} it follows that for almost every $x\in I_0$, the sequence $\big(g(x)^{b(n)},a(n)f(x)\big)$ is u.d. on $\T^2$.

    Since $f$ is analytic and non-constant, the set $\{x:f'(x)=0\}$ is discrete, and hence it has zero measure and its complement can be covered with countably many compact intervals $I_0$.
    The previous paragraph shows for each such interval $I_0$, for almost every $x\in I_0$, the sequence $\big(g(x)^{b(n)},a(n)f(x)\big)$ is u.d. on $\T^2$, and the conclusion follows from the fact that the countable union of sets of measure zero has measure zero.
    \end{proof}

Alas, the ``change of variable'' trick used in the proof of \cref{thm_weyl_koksmafusion} does not extend to longer patterns such as 
$\big(g(x)^{b(n)},a_1(n)f_1(x),a_2(n)f_2(x)\big)$ when $f_1$ and $f_2$ are arbitrary non-constant analytic functions.

\appendix

\section{An oscillatory integral -- by Jim Wright}

In this appendix we prove \cref{thm_wright} from the introduction, whose statement we now recall.

\begin{named}{\cref{thm_wright}}{}
            Let $k\in\N$, let $f_1,\dots,f_k:\R\to\R$ be analytic functions such that $\{1,f_1,\dots,f_k\}$ is linearly independent over $\R$, and let $I\subset\R$ be a compact interval.
            
            Then there exist $C,\delta>0$ such that for every non-zero $\lambda=(\lambda_1,\dots,\lambda_k)\in\R^k$,
            \begin{equation} 
            \left|\int_Ie\big(\lambda_1f_1(x)+\cdots+\lambda_kf_k(x)\big)\d x\right|\leq C\|\lambda\|^{-\delta}.\tag{\ref{eq_wrightestimate}}
            \end{equation}
        \end{named}
\begin{remark}
    The condition that $\{1,f_1,\dots,f_k\}$ is linearly independent over $\R$ is necessary, as otherwise there would exist a non-zero vector $\lambda=(\lambda_1,\dots,\lambda_k)\in\R^k$ such that the function $x\mapsto \lambda_1f_1(x)+\cdots+\lambda_kf_k(x)$ is constant.
Then for any scalar multiple $\tilde\lambda=c\lambda$ of this vector the left hand side of \eqref{eq_wrightestimate} equals $|I|$, but the right hand side tends to $0$ as $c\to\infty$. 
\end{remark}

\begin{proof} Denote by $f:I\to\R^k$ the vector valued function $f(x)=\big(f_1(x),\dots,f_k(x)\big)$ and by ${\mathbb S}^{k-1}=\{\omega\in\R^k:\|\omega\|=1\}$ the unit sphere in $\R^k$. 
For each $\lambda\in\R^k$ consider the map 
$\phi_\lambda:I\to\R$ given by $\phi_\lambda(x)=\lambda\cdot f(x)$.

The assumption that $\{1,f_1,\dots,f_k\}$ is linearly independent over $\R$ implies that the function $\phi_\omega$ is not constant for any $\omega\in{\mathbb S}^{k-1}$.
As $\phi_\omega$ is a real-analytic function, this in turn implies that
for every $(x,\omega) \in I \times {\mathbb S}^{k-1}$, there is an integer $d = d(x,\omega) \ge 1$ such that
\begin{equation*}\label{nonzero}
\phi_\omega^{(d)}(x) \ \not= \ 0.
\end{equation*}

Continuity of the maps $(y,\omega)\mapsto \phi_\omega^{(d)}(y)$ now implies that for every $(y, \omega) \in I \times {\mathbb S}^{k-1}$, there is $d\in\N$, $c(y,\omega)>0$ and an open neighborhood $I_y \times S_{\omega} \subset I\times{\mathbb S}^{k-1}$ of $(y,\omega)$, such that
for all $(x,\theta) \in I_y \times S_{\omega}$, 
\begin{equation}\label{bound-below}
\Bigl| \phi_\theta^{(d)}(x) \Bigr| \ \ge \ c(y,\omega).
\end{equation}
By compactness of $I \times {\mathbb S}^{k-1}$, we can extract a finite sub-cover 
$$
I \times {\mathbb S}^{k-1} \ = \
\bigcup_{(y,\omega)\in {\mathcal M}}  I_y \times S_{\omega}.
$$ 
In particular, for any $\omega_0\in{\mathbb S}^{k-1}$,
\begin{equation}\label{eq_claim2}
    I=\bigcup_{(y,\omega)\in{\mathcal M}\atop \omega_0\in S_\omega}I_y.
\end{equation}

Now fix a non-zero $\lambda \in {\mathbb R}^k$ and write 
$\lambda = \|\lambda\|\cdot \omega_0$ where $\omega_0 \in {\mathbb S}^{k-1}$. 
Denote by $R$ the (finite) set of $y\in I$ for which there is $\omega\in {\mathbb S}^{k-1}$ with $(y,\omega)\in{\mathcal M}$ and $\omega_0\in S_\omega$. 
Note that $|R|\leq|{\mathcal M}|$, and in particular it is controlled in terms of $f$ and $I$ only.
In view of \eqref{eq_claim2} we have
\begin{equation}\label{eq_partitioningtheinterval}
\left|\int_Ie\big(\phi_\lambda(x)\big)\d x\right| 
\leq
\sum_{y\in R} \left|\int_{I_y} \, e\big(\phi_\lambda(x)\big)\d x\right| .
\end{equation}

From \eqref{bound-below} it follows that for each $y\in R$ there exists $d\in\N$ and a constant $c_y > 0$ 
such that
\begin{equation}\label{bound-below-0}
 \Bigl| \phi_{\lambda}^{(d)}(x) \Bigr| \ = \ \Bigl| \|\lambda\|\cdot\phi_{\omega_0}^{(d)}(x) \Bigr| \ \ge \ \|\lambda\|\cdot c_y \ \ge \ \|\lambda\|\cdot c
\end{equation}
for all $x \in I_y$, where $c:=\min_{(y,\omega)\in{\mathcal M}}c(y,\omega)$ is positive (because ${\mathcal M}$ is finite) and depends only on $f$ and $I$.

If $d\ge 2$, we combine \eqref{bound-below-0} with \cref{lemma_highordervdC} to deduce that
\begin{equation}\label{eq_whend>1}
\left|\int_{I_y} \, e\big(\phi_\lambda(x)\big)\d x\right|
\leq
C_d\big(\|\lambda\|\cdot c\big)^{-1/d}
\ll\|\lambda\|^{-1/d},
\end{equation}
where the implicit constant may depend on $f$ and $I$ but not on $\lambda$.

If $d=1$, we integrate by parts, writing $e\big(\phi_\lambda(x)\big)=\frac{1}{2\pi i}
\frac{d}{dx} \Bigl[
e\big(\phi_\lambda(x)\big) \Bigr] \frac{1}{\phi_\lambda'(x)}$ and hence
$$
\int_{I_y} e\big(\phi_\lambda(x)\big) \d x 
= \frac{1}{2\pi i}
\left[\int_{I_y} e\big(\phi_\lambda(x)\big) \frac{\phi_\lambda''(x)}{\big(\phi_\lambda'(x)\big)^2} \d x
+\left.\left(e\big(\phi_\lambda(x)\big) \frac{1}{\phi_\lambda'(x)}\right)\right|_{I_y}\right].
$$
Combined with \eqref{bound-below-0}, this leads to the estimate
\begin{equation}\label{eq_whend=1}
    \left|\int_{I_y} e\big(\phi_\lambda(x)\big) \d x \right|
\leq
\frac{|I|}{2\pi}\cdot \sup_I|\phi''_\lambda(x)|\cdot\frac1{(\|\lambda\|c)^2}+\frac1{\pi\|\lambda\|c}
\ll\|\lambda\|^{-1}
\end{equation}
where the implicit constant may depend on $f$ and $I$ but not on $\lambda$.
Since \eqref{eq_wrightestimate} holds trivially when $\|\lambda\|\leq 1$ by taking $C>|I|$, we may assume that $\|\lambda\|>1$.
In this case the value of $\|\lambda\|^{-1/d}$ increases with $d$ and so we can deduce from \eqref{eq_whend>1} and \eqref{eq_whend=1} that
\begin{equation}\label{eq_foranyd}
    \left|\int_{I_y} e\big(\phi_\lambda(x)\big) \d x \right|
\ll
\|\lambda\|^{-\delta}
\end{equation}
holds for any $\lambda\in\R^k$ with $\|\lambda\|>1$ and any $y\in R$; where $\delta:=1/\max_{(y,\omega)\in{\mathcal M}}d(y,\omega)$ is independent of $\lambda$.

Putting \eqref{eq_foranyd} and \eqref{eq_partitioningtheinterval} together, and keeping in mind that the cardinality of $R$ is bounded above by $|{\mathcal M}|$, which does not depend on $\lambda$, we obtain \eqref{eq_wrightestimate}, finishing the proof.
\end{proof}

\bibliographystyle{plain}
\bibliography{refs-joel}
\end{document}